\documentclass{article}

\usepackage{graphicx} 
\usepackage{amsfonts}
\usepackage{amssymb}
\usepackage{amsthm}
\usepackage{amsmath}
\usepackage{color}

\newtheorem{thm}{Theorem}[section]
\newtheorem{lem}[thm]{Lemma}
\newtheorem{prop}[thm]{Proposition}
\newtheorem{cor}[thm]{Corollary}

\newtheorem{defn}{Definition}

\usepackage{amsfonts}
\usepackage{amssymb}
\usepackage{amsthm}
\usepackage{amsmath}
\begin{document}

\title{Computable Scott Sentences and the Friedman-Stanley embedding}

\author{David Gonzalez and Julia Knight}

\maketitle

\begin{abstract}

Friedman and Stanley \cite{FS} developed the notion of \emph{Borel reducibility} and illustrated its use in comparing classification problems for some familiar classes of countable structures.  For many embeddings, the fact that the embedding is $1-1$ on isomorphism types is explained by the existence of simple formulas that, uniformly, interpret the input structure in the output structure.  For the embeddings of graphs in trees, and in linear orderings, there is no uniform interpretation \cite{H-TM}, \cite{KSV}.  We focus on a version of the Friedman-Stanley embedding from \cite{H-TM} that takes each structure $\mathcal{A}$ for the language of graphs to a labeled tree $T_\mathcal{A}$.  Gonzalez and Rossegger \cite{GR} showed that this embedding preserves Scott complexity.  We refine this result, showing that for an $X$-computable ordinal, if one of $\mathcal{A}$, $T_\mathcal{A}$ has a computable infinitary Scott sentence, then so does the other, and the complexities match.  Let $\mathbb{T}$ be the class of labeled trees isomorphic to those in the range of the embedding, and let $\mathbb{T}^\alpha$ be the subclass consisting of structures of Scott rank at most $\alpha$.  It follows from results of Gao \cite{gao} that $\mathbb{T}$ is not Borel.  We show that for each $\alpha$, $\mathbb{T}^\alpha$ is Borel.  In fact, if $\alpha$ is an $X$-computable ordinal, then $\mathbb{T}^\alpha$  is complete $X$-effective $\Pi_{2\alpha+2}$.    
\end{abstract}    

\section{Introduction}

A good deal of work, in different branches of mathematics, involves trying to ``classify'' the objects in some class up to some notion of sameness, using relatively simple invariants.  For a countable language $L$, $Mod(L)$ denotes the class of $L$-structures with universe $\omega$.  Our objects are elements of $Mod(L)$, for some countable language $L$.  Our classes are subclasses of $Mod(L)$, always closed under isomorphism.  Our notion of sameness is isomorphism.  The class of $\mathbb{Q}$-vector spaces is nicely classified by dimension.  Nobody would try to classify directed graphs.  Of course, there are only countably many isomorphism types of $\mathbb{Q}$-vector spaces, while there are $2^{\aleph_0}$ isomorphism types of countable directed graphs. Friedman and Stanley \cite{FS} developed the notion of ``Borel reducibility'' to compare classification problems for classes of countable structures, using more than just cardinality.    

\begin{defn} 

For classes $K\subseteq Mod(L)$ and $K'\subseteq Mod(L')$, we say that $K$ is \emph{Borel reducible} to $K'$, and we write $K\leq_B K'$, if there is a Borel function $\Phi:K\rightarrow K'$ such that for $\mathcal{A},\mathcal{B}\in K$, $\mathcal{A}\cong\mathcal{B}$ iff $\Phi(\mathcal{A})\cong\Phi(\mathcal{B})$.  

\end{defn}

In \cite{FS}, Friedman and Stanley located various familiar classes of structures under Borel reducibility, showing that undirected graphs, fields, $2$-step nilpotent groups, and linear orderings lie on top, while $\mathbb{Q}$-vector spaces and Abelian $p$-groups lie strictly below the top.   There has been a great deal of further work locating further interesting classes.  For the class $TFAb_n$ of torsion-free Abelian groups of rank $n$, Hjorth \cite{Hjorth} and Thomas \cite{Thomas} showed that $TFAb_n <_B TFAb_{n+1}$.  Paolini and Shelah \cite{PS} showed that torsion-free Abelian groups of infinite rank lie on top.  For some of their results, Friedman and Stanley used known embeddings.  To show that linear orderings and trees lie on top, they gave new embeddings.  We focus on a version of the Friedman-Stanley embedding, described in \cite{H-TM}, that, for a finite relational language $L$, takes each $\mathcal{A}\in Mod(L)$ to a labeled tree $T_\mathcal{A}$. 

\bigskip

Let $\mathcal{A}\in Mod(L)$, where $L$ is the language of graphs, and let $T_\mathcal{A}$ be the corresponding tree structure.  In \cite{GR}, it is shown that $\mathcal{A}$ and $T_\mathcal{A}$ have the same Scott rank.  Hence, if one of the structures has a $\Pi_{\alpha+1}$ Scott sentence, then so does the other.  They had asked whether the same is true for computable infinitary Scott sentences.   In Section 2, we give a positive answer.

\bigskip
\noindent
\textbf{Theorem A}. For all $\mathcal{A}\in Mod(L)$, if one of $\mathcal{A}$, $T_\mathcal{A}$ has a computable $\Pi_{\alpha+1}$ Scott sentence, then so does the other.  

\bigskip
\noindent
\textbf{Remark}:  Our proof shows more. Analagously to the work in \cite{GR}, we see that, if $\alpha$ is a computable limit ordinal and $\mathcal{A}$ has a computable $\Pi_\alpha$ Scott sentence, then so does $T_\mathcal{A}$.

\bigskip

Let $\mathbb{T}$ be the class of labeled trees isomorphic to one of the form $T_\mathcal{A}$ for $\mathcal{A}\in Mod(L)$.  Let $\mathbb{T}^\alpha$ consist of the elements of $\mathbb{T}$ of rank at most $\alpha$.  The full class $\mathbb{T}$ is not Borel.  Gao \cite{gao} gave results implying that for any embedding of graphs in linear orderings, the class of orderings isomorphic to those in the range is not Borel.  The same reasoning shows that $\mathbb{T}$ is not Borel (see \cite{GR} or \cite{GHT}).  In Section 3, we show that for all countable ordinals $\alpha$, $\mathbb{T}^\alpha$ \emph{is} Borel, and we give the precise complexity. In Section 4, we show that our given complexity is the best possible.

\bigskip
\noindent
\textbf{Theorem B}.  For all $X$-computable ordinals $\alpha\geq 1$, $\mathbb{T}^\alpha$ is complete $X$-effective $\Pi_{2\alpha+2}$.  

\subsection{Background}

In the remainder of the introduction, we give some background.  We describe the embedding $\mathcal{A}\rightarrow T_\mathcal{A}$ precisely.  We recall infinitary formulas, $L_{\omega_1\omega}$-formulas, and computable infinitary formulas.  We state basic results of Scott \cite{S} and Montalb\'{a}n \cite{robuster} on Scott sentences and Scott rank.  We define the standard back-and-forth relations.  Finally, we describe variants of Borel reducibility for which a Pullback Theorem holds \cite{KMV}.   

\subsubsection{Tree of tuples} 

In this subsection, we describe the embedding $\mathcal{A}\rightarrow T_\mathcal{A}$ in a precise way, following \cite{H-TM}, see also \cite{GR}, p.\ 4.  Let $L$ be a finite relational language.  Let $X$ be the set of variables $x_1,x_2,\ldots$.  We assign to each $L$-structure $\mathcal{A}$, with universe $\omega$, a tree structure $T_\mathcal{A}$.  The language of the tree structures, $L^*$, consists of a unary function symbol $p$ (the predecessor function) and unary relation symbols $U_D$, where $D$ assigns an atomic type to the variables $x_1,\ldots,x_n$ for some $n$, making the variables all distinct.    

For each tuple $\bar{a} = (a_1,\ldots,a_n)$ of distinct elements of $\mathcal{A}$, let $\bar{x} = (x_1,\ldots,x_n)$ be the corresponding tuple of variables from $X$, and let $D_{\bar{a}}(\bar{x})$ be the conjunction of the formulas $\pm\alpha(\bar{x})$ true of $\bar{a}$, where $\alpha(\bar{x})$ is atomic. Each node of $T_\mathcal{A}$ represents some tuple $\bar{a}$ and carries exactly the label $U_D$, where $D = D_{\bar{a}}(\bar{x})$ for the appropriate initial tuple of variables $\bar{x}$. We ensure that every possible tuple from $\mathcal{A}$ is represented among the nodes of $T_\mathcal{A}$. A predecessor of a node will represent the tuple with one less element.  We define $T_\mathcal{A}$ more precisely below.            

\begin{defn} [Tree of tuples $T_\mathcal{A}$]

At level $0$, $T_\mathcal{A}$ has a unique node $\lambda$, representing $\emptyset$. The nodes at level $n$ represent tuples $\bar{a}$ of length $n$. A node representing a tuple $\bar{a}$ is labeled with a code for the atomic diagram of $\bar{a}$. For each node $\sigma$ representing an $n$-tuple $\bar{a}$, and each further element $a'$, $\sigma$ has infinitely many extensions representing the tuple $(\bar{a},a')$.         

\end{defn}

\subsubsection{Formulas of $L_{\omega_1\omega}$}

The infinitary logic $L_{\omega_1\omega}$ allows countably infinite disjunctions and conjunctions, but only finite strings of quantifiers.  Bringing negations inside, we get a normal form in which $\bigvee\exists$ alternates with $\bigwedge\forall$.  We classify formulas in this normal form as $\Sigma_\alpha$ or $\Pi_\alpha$ for countable ordinals $\alpha$.

\begin{enumerate}
    
\item $\varphi(\bar{x})$ is $\Sigma_0$ and $\Pi_0$ if it is finitary quantifier-free.

\item  For $\alpha\geq 1$, 

\begin{enumerate} 
    
\item  $\varphi(\bar{x})$ is $\Sigma_\alpha$ if it has form $\bigvee_i\exists\bar{u}_i\psi_i(\bar{x},\bar{u}_i)$, where each $\psi_i$ is $\Pi_{\beta_i}$ for some $\beta_i < \alpha$,   

\item  $\varphi(\bar{x})$ is $\Pi_\alpha$ if it has form $\bigwedge_i\forall\bar{u}_i\psi_i(\bar{x},\bar{u}_i)$, where each $\psi_i$ is $\Sigma_{\beta_i}$ for some $\beta_i < \alpha$.

\end{enumerate}
\end{enumerate}

\subsubsection{Computable infinitary formulas}

Computable infinitary formulas are $L_{\omega_1\omega}$ formulas in which the infinite disjunctions and conjunctions are over c.e.\ sets.  To make this precise, we need to assign indices to the formulas, as in \cite{AKbook} or \cite{Mbook2}.  We classify computable infinitary formulas as computable $\Sigma_\alpha$ or computable $\Pi_\alpha$ for computable ordinals $\alpha$.   

\subsubsection{Scott complexity}

In this subsubsection, we recall the Scott Isomorphism Theorem \cite{S} and define a Scott sentence.  We also state a result of Montalb\'{a}n \cite{robuster}, which yields the notion of Scott rank we will use.   

\begin{thm} [Scott Isomorphism Theorem]

For each countable structure $\mathcal{A}$ for a countable language $L$, there is a sentence $\varphi$ of $L_{\omega_1\omega}$ whose countable models are just the isomorphic copies of $\mathcal{A}$.

\end{thm}

\begin{defn}
    
A \emph{Scott sentence} for $\mathcal{A}$ is a sentence $\varphi$ that characterizes $\mathcal{A}$ up to isomorphism, as in Scott's Theorem. 

\end{defn}

\begin{thm} [Montalb\'{a}n]

Let $\alpha\geq 1$ be a countable ordinal.  For $\mathcal{A}\in Mod(L)$, the following are equivalent:

\begin{enumerate}

\item  $\mathcal{A}$ has a $\Pi_{\alpha+1}$ Scott sentence,

\item  for each tuple $\bar{a}$, the orbit is defined by a $\Sigma_\alpha$ formula,

\item  for each tuple $\bar{a}$, there is a $\Sigma_\alpha$-formula that implies (in $\mathcal{A}$) all $\Pi_\alpha$ formulas satisfied by $\bar{a}$.

\end{enumerate}

\end{thm}

\begin{defn}

The \emph{Scott rank} of $\mathcal{A}$, denoted by $SR(\mathcal{A})$, is the least $\alpha$ such that the orbits of all tuples are defined by computable $\Sigma_\alpha$ formulas.  If $\alpha$ is a limit ordinal, then for a structure $\mathcal{A}$ of Scott rank $\alpha$, the orbit of each tuple is defined by a $\Sigma_\beta$ formula for some $\beta < \alpha$.  The structure $\mathcal{A}$ will have a $\Pi_{\alpha+1}$ Scott sentence.  It may or may not have a Scott sentence that is $\Pi_\alpha$.   

\end{defn}

Alvir, Knight, and McCoy \cite{AKM} gave the following partial effective version of Montalb\'{a}n's result.

\begin{thm}\label{thm:AKM}

If $\mathcal{A}$ has a computable $\Pi_\alpha$ Scott sentence, then the orbit of each tuple is defined by a computable $\Sigma_\beta$ formula for some $\beta < \alpha$.    

\end{thm}

\subsubsection{Back-and-forth relations}

We define the standard back-and-forth relations  $\leq_\alpha$.  We are mainly interested in tuples from a single structure, although in the definition, we allow tuples from two different structures.       

\begin{defn}

Let $\mathcal{A},\mathcal{B}$ be structures for the same countable language.  First, suppose that $\bar{a}$ in $\mathcal{A}$ and $\bar{b}$ in $\mathcal{B}$ are tuples of the same length.  

\begin{enumerate}

\item  $(\mathcal{A},\bar{a})\leq_1(\mathcal{B},\bar{b})$ if all existential formulas satisfied by $\bar{b}$ in $\mathcal{B}$ are satisfied by $\bar{a}$ in $\mathcal{A}$,

\item  For $\alpha > 1$, $(\mathcal{A},\bar{a})\leq_\alpha (\mathcal{B},\bar{b})$ if for each $\bar{d}$ and each $1\leq\beta < \alpha$, there exists $\bar{c}$ such that $(\mathcal{B},\bar{b},\bar{d})\leq_\beta(\mathcal{A},\bar{a},\bar{c})$.  

\end{enumerate}  

If $\bar{b}$ is longer than $\bar{a}$, where $\bar{b}'$ is the initial segment of $\bar{b}$ of the same length as $\bar{a}$, then $(\mathcal{A},\bar{a})\leq_\alpha(\mathcal{B},\bar{b})$ if $(\mathcal{A},\bar{a})\leq_\alpha(\mathcal{B},\bar{b}')$.   

\end{defn}

Carol Karp \cite{Karp} proved the following.

\begin{thm}

Let $\bar{a}$ in $\mathcal{A}$ and $\bar{b}$ in $\mathcal{A}$ be tuples of the same length.  For all countable ordinals $\alpha\geq 1$, the following are equivalent:

\begin{enumerate}

\item  $(\mathcal{A},\bar{a})\leq_\alpha(\mathcal{B},\bar{b})$,

\item  all $\Pi_\alpha$ formulas satisfied by $\bar{a}$ in $\mathcal{A}$ are satisfied by $\bar{b}$ in $\mathcal{B}$,

\item  all $\Sigma_\alpha$ formulas satisfied by $\bar{b}$ in $\mathcal{B}$ are satisfied by $\bar{a}$ in $\mathcal{A}$.

\end{enumerate}

\end{thm}  

The following is well-known---see \cite{AK}.  

\begin{lem}
\label{Pialpha}

Let $L$ be a countable language, and let $\mathcal{A}$ be a countable $L$-structure.  For each countable ordinal $\alpha$ and each tuple $\bar{a}$ in $\mathcal{A}$, there is a $\Pi_\alpha$-formula $\varphi(\bar{x})$ that defines in $\mathcal{A}$ the set of $\bar{b}$ such that $\bar{a}\leq_\alpha\bar{b}$.

\end{lem}

\begin{proof} [Proof sketch]

For each tuple $\bar{b}$ such that $\bar{a}\not\leq_\alpha\bar{b}$, choose a $\Pi_\alpha$ formula true of $\bar{a}$ and not true of $\bar{b}$.  The conjunction of the chosen formulas is the desired $\varphi(\bar{x})$.
\end{proof}

We add one more to the list of conditions in Montalb\'{a}n's Theorem.   

\begin{lem}
\label{addedcondition}

Let $\mathcal{A}$ be a countable structure for a countable language $L$.  For a countable ordinal $\alpha$, the following are equivalent:

\begin{enumerate}

\item  $\mathcal{A}$ has Scott rank at most $\alpha$

\item for each tuple $\bar{a}$, there exist $\beta < \alpha$ and a tuple $\bar{b}$ such that for all tuples $\bar{a}'$ of the same length as $\bar{a}$ and $\bar{b}'$ of the same length as $\bar{b}$, $\bar{a},\bar{b}\leq_\beta\bar{a}'\bar{b}'$ implies $\bar{a}\leq_\alpha\bar{a}'$.

\end{enumerate}

\end{lem}

\begin{proof}

$(1)\Rightarrow(2)$:  For each tuple in $\mathcal{A}$, there is a $\Sigma_\alpha$ formula $\varphi(\bar{x})$ that defines the orbit of $\bar{a}$.  We may suppose that $\varphi(\bar{x})$ has form $(\exists\bar{u})\psi(\bar{x},\bar{u})$, where $\psi(\bar{x},\bar{u})$ is $\Pi_\beta$ for some $\beta < \alpha$.  Take $\bar{b}$ such that $\mathcal{A}\models\psi(\bar{a},\bar{b})$.  If $\bar{a},\bar{b}\leq_\beta\bar{a}',\bar{b}'$, then $\mathcal{A}\models\psi(\bar{a}',\bar{b}')$.  Then $\bar{a}\leq_\alpha\bar{a}'$---the two tuples are automorphic.    

$(2)\Rightarrow(1)$:  Take $\bar{a}$ in $\mathcal{A}$.  By hypothesis, we have $\beta < \alpha$ and $\bar{b}$ such that $\bar{a},\bar{b}\leq_\beta \bar{a}',\bar{b}'$ implies $\bar{a}\leq_\alpha\bar{a}'$.  By Lemma \ref{Pialpha}, there is a $\Pi_\alpha$-formula $\psi(\bar{x},\bar{u})$ such that $\mathcal{A}\models\psi(\bar{a}',\bar{b}')$ iff $\bar{a},\bar{b}\leq_\beta\bar{a}',\bar{b}'$.  Then the $\Sigma_\alpha$-formula $(\exists\bar{u})\psi(\bar{x},\bar{u})$ implies all $\Pi_\alpha$ formulas true of $\bar{a}$, so it defines the orbit.      
\end{proof}

\subsubsection{Embeddings that preserve complexity}

We consider variants of the notion of Borel embedding that satisfy a ``Pullback'' theorem.  
The following definition is from \cite{CCKM}.

\begin{defn} [Turing computable embedding]

Let $L$ and $L'$ be computable languages and take classes $K\subseteq Mod(L)$, $K'\subseteq Mod(L')$.   A \emph{Turing computable embedding} of $K$ in $K'$ is a Turing operator $\Phi:K\rightarrow K'$ such that $\mathcal{A}\cong\mathcal{B}$ iff $\Phi(\mathcal{A})\cong\Phi(\mathcal{B})$.  We write $K\leq_{tc} K'$ if there is such an embedding.  

\end{defn}

The result below is from \cite{KMV}. 

\begin{thm} [Pull-back Theorem]

Suppose $K\leq_{tc} K'$, where $K\subseteq Mod(L)$ and $K'\subseteq Mod(L')$.  For any computable infinitary $L'$-sentence  $\varphi$, we can effectively find a computable infinitary $L$-sentence $\varphi^*$, of the same complexity, such that for all $\mathcal{A}\in K$, $\Phi(\mathcal{A})\models \varphi$ iff $\mathcal{A}\models\varphi^*$.  

\end{thm}

The notion of Turing computable embedding relativizes.  

\bigskip
\noindent
\textbf{Fact}:  An embedding is continuous iff it is $X$-computable for some $X$.  

\bigskip

The Pull-back Theorem relativizes to give the following.     

\begin{cor}

Let $\Phi$ be a continuous embedding of $K$ in $K'$, where $K\subseteq Mod(L)$ and $K'\subseteq Mod(L')$.  Then for any $L_{\omega_1\omega}$ sentence $\varphi$, there is an $L_{\omega_1\omega}$ sentence $\varphi^*$, of the same complexity, such that $\Phi(\mathcal{A})\models\varphi$ iff $\mathcal{A}\models\varphi^*$.  

\end{cor}

The following is well-known; see, for example \cite{Mbook1} Theorem VI.27 and Lemma VI.29.

\begin{thm}

Let $L$ be a computable relational language.  There is a Turing computable embedding $\Phi$ of $Mod(L)$ in the class of undirected graphs.  Moreover, 

\begin{enumerate}

\item  the class $K$ consisting of graphs $G$ such that for some $\mathcal{A}\in Mod(L)$, $G\cong\Phi(\mathcal{A})$ is effective $\Pi_2$, 

\item  there are existential formulas that for all $\mathcal{A}\in Mod(L)$, define a copy of $\mathcal{A}$ in $\Phi(\mathcal{A})$; we have formulas $d_U(x)$, defining the universe of the copy, and, for $R\in L$, $d_R(\bar{x})$ defining the interpretation of $R$, and $d_{\neg{R}}(\bar{x})$ defining the interpretation of $\neg{R}$.  

\end{enumerate}
Hence, there is a Turing computable embedding $\Psi$ of $K$ in $Mod(L)$ such that for all $\mathcal{A}\in Mod(L)$, $\Psi(\Phi(\mathcal{A}))\cong\mathcal{A})$.  

\end{thm} 

Using this theorem, we easily obtain the following.

\begin{prop} 

Let $L$ be a computable relational language.  There is a Turing computable embedding $\Phi$ of $Mod(L)$ in the class of graphs such that for $\mathcal{A}\in Mod(L)$, the Scott sentences for $\mathcal{A}$ and $\Phi(\mathcal{A})$ have the same complexity.

\end{prop}

\begin{proof}

Let $\Phi$, $K$, and $\Psi$ be as in the previous theorem.  Take $\mathcal{A}\in Mod(L)$, and let $G = \Phi(\mathcal{A})$.  First, let $\varphi$ be a Scott sentence for $G$.  Let $\varphi^*$ be the $\Phi$-pullback of $\varphi$, of the same complexity as $\varphi$.  For $\mathcal{A}'\in Mod(L)$, we have $\mathcal{A}'\models\varphi^*$ iff $\Phi(\mathcal{A}')\models\varphi$ iff $\Phi(\mathcal{A}')\cong G$ iff $\mathcal{A}'\cong\mathcal{A}$.  Therefore, $\varphi^*$ is a Scott sentence for $\mathcal{A}$.

Now, let $\psi$ be a Scott sentence for $\mathcal{A}$.  Let $\psi^*$ be the $\Psi$-pullback of $\psi$.  Adding a conjunct, if necessary, we may suppose that $\psi^*$ implies the computable $\Pi_2$ sentence characterizing the class $K$.  Then we can show, by a sequence of equivalences just like the one in the previous paragraph, that $\psi^*$ is a Scott sentence for $G$. 
\end{proof}

\section{Transfer of computable Scott sentences}

In this section, we prove the first of our two main results.  In \cite{GR}, Gonzalez and Rossegger proved that for a structure $\mathcal{A}\in Mod(L)$, $\mathcal{A}$ and $T_\mathcal{A}$ have the same Scott rank. By Montalb\'{a}n's Theorem, if one of the structures has a Scott sentence that is $\Pi_{\alpha+1}$, then so does the other.  Gonzalez and Rossegger asked whether the same is true for computable infinitary Scott sentences.  We give a positive answer. 

\bigskip
\noindent
\textbf{Theorem A}.  Let $L$ be a finite relational language.  If one of the structures $\mathcal{A}$ or $T_\mathcal{A}$ has a computable $\Pi_{\alpha+1}$ Scott sentence, then so does the other.  Moreover, we can pass effectively from a Scott sentence for one to a Scott sentence for the other.      

\bigskip
\noindent
\textbf{Remark}:  We can vary Theorem A, letting the language $L$ be computable, but not necessarily finite.  

\bigskip

We split the proof of Theorem A in two.  The result below gives the easier implication.   

\begin{prop}
\label{Prop7}

Given a computable infinitary Scott sentence $\varphi$ for $T_\mathcal{A}$, we can effectively pass to a computable infinitary Scott sentence $\varphi^*$ for $\mathcal{A}$, where the complexity of $\varphi^*$ matches that of $\varphi$.  

\end{prop}

\begin{proof}

The embedding $\mathcal{A}\rightarrow T_\mathcal{A}$ is Turing computable.  Applying the Pullback Theorem, we take $\varphi^*$ to be the pullback of $\varphi$.  The two sentences have the same complexity, and for $\mathcal{B}\in Mod(L)$, $\mathcal{A}\cong\mathcal{B}$ iff $T_\mathcal{A}\cong T_\mathcal{B}$.  Since $\varphi$ is a Scott sentence for $T_\mathcal{A}$, we have that
\[\mathcal{B}\models\varphi^*\iff T_\mathcal{B}\models\varphi\iff T_\mathcal{B}\cong T_\mathcal{A}\iff \mathcal{A}\cong\mathcal{B},\]
and so $\varphi^*$ is a Scott sentence for $\mathcal{A}$.          
\end{proof} 

Every $L_{\omega_1\omega}$-sentence is $X$-computable for some $X$.  Relativizing Proposition~\ref{Prop7}, we get the result, already proved by Gonzalez and Rossegger \cite{GR}, saying that if $T_\mathcal{A}$ has a Scott sentence that is $\Pi_{\alpha}$,  then so does $\mathcal{A}$. 

\bigskip

The result below gives the harder implication in Theorem A. 

\begin{prop}\label{prop:PushForwardCompSentences}

  For a computable ordinal $\alpha$, if $\mathcal{A}$ has a computable $\Pi_\alpha$ Scott sentence $\varphi$, then so does $T_\mathcal{A}$.  

\end{prop} 

\subsection{Outline for the proof of Proposition \ref{prop:PushForwardCompSentences}}  

We begin by isolating some simple properties of trees $T_\mathcal{A}$ for $\mathcal{A}\in Mod(L)$.  For a labeled tree $T$ satisfying the first two properties below, any path $p$ yields a structure $\mathcal{A}^p$.  Next, we define forcing and prove the basic lemmas as in Cohen \cite{Cohen}.  Suppose $T\cong T_\mathcal{A}$, where $\mathcal{A}$ has computable Scott rank $\alpha$.  We show that if $p$ is a generic path, then $\mathcal{A}^p\cong\mathcal{A}$.  If $\mathcal{A}$ has a computable infinitary Scott sentence $\varphi$, then the lemma on definability of forcing lets us pass from $\varphi$ to a computable infinitary Scott sentence for $T_\mathcal{A}$.  

\subsection{Simple properties of trees $T_\mathcal{A}$}

The basic properties listed below are simple to state, and they are clearly true of the tree structures $T_\mathcal{A}$.  We will give further properties after we have defined forcing.  

\bigskip
\noindent
\textbf{Levels}:  Each node satisfies exactly one $U_D$.  Moreover, if $\sigma$ is at level $n$, then $\sigma$ satisfies $U_D$ for some $D$ assigning a complete atomic type to the variables  $(x_1,\ldots,x_n)$ of $\sigma$.   

\bigskip
\noindent
\textbf{Consistency}:  If $\tau\supseteq\sigma$, then the atomic type assigned to $\tau$ implies the one assigned to $\sigma$.

\bigskip
\noindent
\textbf{Replication}:  For any $x$, if $x$ has at least one successor satisfying $U_D$, then it has infinitely many.    

\subsection{The Definition of Forcing}

In this subsection, we define the appropriate kind of forcing to produce a generic path.  Let $T$ be a tree satisfying the Levels and Consistency properties.  Each path $p$ yields an $L$-structure $\mathcal{A}^p$ with universe $X = \{x_1,x_2,x_3,\ldots\}$.  We think of the variables in $X$ as constants.  The atomic diagram of $\mathcal{A}^p$ is the set of conjuncts $\pm\alpha(\bar{x})$ of formulas $D = D(\bar{x})$ such that some $\sigma\prec p$ has label $U_D$.  The forcing conditions are the nodes of $T$.  The partial ordering is the tree ordering $\preceq$.   The forcing language describes the structure $\mathcal{A}^p$ obtained from a generic path $p$.  The language consists of computable infinitary sentences $\psi(\bar{x})$ in the language $L\cup X$.  We define forcing for these sentences by induction on complexity.    

\begin{defn}\

\begin{enumerate}

\item  Suppose $\alpha(\bar{x})$ is finitary quantifier-free.  For $\sigma$ with label $U_D$, $\sigma\Vdash\alpha$ if the variables $\bar{x}$ are among those of $\sigma$ and $D$ logically implies $\alpha(\bar{x})$.  

\item  Suppose $\theta(\bar{x})$ is computable $\Sigma_\beta$, of form $\bigvee_i(\exists\bar{u}_i)\psi_i(\bar{x},\bar{u}_i)$, where $\psi_i$ is computable $\Pi_{\gamma_i}$ for some $\gamma_i < \beta$, Then $\sigma\Vdash\theta(\bar{x})$ if for some $i$ and some assignment $t$ mapping 
$\bar{u}_i$ to variables of $\sigma$, we have $\sigma\Vdash\psi_i(\bar{x},t(\bar{u}_i))$.

\item  Suppose $\theta(\bar{x})$ is computable $\Pi_\beta$, of form $\bigwedge_i(\forall\bar{u}_i)\psi_i(\bar{x},\bar{u}_i)$.  Then $\sigma\Vdash\theta(\bar{x})$ if the variables $\bar{x}$ are among those of $\sigma$ and for all $\sigma'\supseteq\sigma$, $\sigma'\not\Vdash neg(\theta(\bar{x})$---recall that $neg(\theta)$ is the (computable $\Sigma_\beta$) formula that is equivalent to $\neg{\theta}$, but with the negations brought inside.    

\end{enumerate}

\end{defn}   

\subsection{The Forcing Lemmas}

We state the usual forcing lemmas.  We show that they hold for all $T$ satisfying the Levels and Consistency properties.       

\begin{lem} [Extension]  

If $\sigma\Vdash\psi(\bar{x})$ and $\sigma'\supseteq\sigma$, then $\sigma'\Vdash\psi(\bar{x})$.

\end{lem}

\begin{proof}

First, suppose $\psi(\bar{x})$ is finitary quantifier-free.  If $\sigma\Vdash\psi(\bar{x})$, then the variables of $\sigma$ include $\bar{x}$, and the label on $\sigma$ assigns an atomic type $D$ that logically implies $\psi(\bar{x})$.  If $\sigma'\supseteq\sigma$, then by Consistency, the label on $\sigma'$, assigning an atomic type $D'$ to a larger tuple of variables, implies $D$, so it implies $\psi(\bar{x})$.
Next, suppose $\psi(\bar{x}) = \bigvee_i(\exists\bar{u}_i)\psi_i(\bar{x},\bar{u}_i)$, where the statement holds for all $\psi_i(\bar{x},\bar{z})$.  If $\sigma\Vdash\psi(\bar{x})$, then for some $i$ and some assignment fixing $\bar{x}$ and taking $\bar{u}_i$ to some $\bar{z}$, $\sigma\Vdash\psi_i(\bar{x},\bar{z})$.  If $\sigma'\supseteq\sigma$, then by the Induction Hypothesis, $\sigma'\Vdash\psi_i(\bar{x},\bar{z}),$ so $\sigma\Vdash\psi(\bar{x})$.  
Finally, suppose $\psi(\bar{x}) = \bigwedge_i(\forall\bar{u}_i)\psi_i(\bar{x},\bar{u}_i)$, where the statement holds for all $\psi_i(\bar{x},\bar{z})$.  If $\sigma\Vdash\psi(\bar{x})$, then by the definition of forcing, no extension of $\sigma$ forces $neg(\psi(\bar{x}))$.  For $\sigma'\supseteq\sigma$, no extension of $\sigma'$ forces $neg(\psi(\bar{x}))$, so $\sigma'\Vdash\psi(\bar{x})$.    
\end{proof}

\begin{lem} [Consistency]  

For all $\sigma$ and $\psi(\bar{x})$, $\sigma$ cannot force both $\psi(\bar{x})$ and $neg(\psi(\bar{x}))$. 

\end{lem}

\begin{proof}

First, suppose $\psi(\bar{x})$ is finitary quantifier-free.  Then $neg(\psi(\bar{x})) = \neg{\psi(\bar{x})}$.  If $\sigma$ assigns the atomic type $D$ to variables including $\bar{x}$, $D$ does not imply both $\psi(\bar{x})$ and $\neg{\psi(\bar{x})}$, so $\sigma$ does not force both.  
Next let $\psi(\bar{x}) = \bigvee_i(\exists\bar{u}_i)\psi_i(\bar{x},\bar{u}_i)$.  If $\sigma\Vdash neg(\psi(\bar{x}))$, then by definition, $\sigma$ does not force $\psi(\bar{x})$.
Similarly, for $\psi(\bar{x}) = \bigwedge_i(\forall\bar{u}_i)\psi_i(\bar{x},\bar{u}_i)$, if $\sigma\Vdash\psi(\bar{x})$, then by definition, $\sigma$ does not force $\psi(\bar{x})$.  
\end{proof}

\begin{lem}[Density]
\label{lem:density}   

For all $\sigma$ and $\psi(\bar{x})$, some extension of $\sigma$ forces either 
$\psi(\bar{x})$ or $neg(\psi(\bar{x}))$. 

\end{lem}

\begin{proof}

First, suppose $\psi(\bar{x})$ is finitary quantifier-free.  Take $\sigma'\supseteq\sigma$ such that $\bar{x}$ is included among the variables of $\sigma'$.  Then $\sigma'$ forces one of $\psi(\bar{x})$, $\neg(\psi(\bar{x}))$, as the label of $\sigma'$ describes the full quantifier-free type of $\bar{x}$.  
Next, suppose $\psi(\bar{x}) = \bigvee_i(\exists\bar{u}_i)\psi_i(\bar{x},\bar{u}_i)$, where the statement holds for all $\psi_i(\bar{x},\bar{z})$.  If no $\sigma'\supseteq\sigma$ forces $\psi(\bar{x})$, then by definition, $\sigma$ forces $neg(\psi(\bar{x}))$.  
Similarly, if $\psi(\bar{x}) = \bigwedge_i(\forall\bar{u}_i)\psi_i(\bar{x},\bar{u}_i)$, and $\sigma$ does not force $\psi(\bar{x})$, then by definition, some $\sigma'\supseteq\sigma$ forces $neg(\psi(\bar{x}))$.    
\end{proof}

Using the Density and Extension Lemmas, we get the following.

\begin{lem} [Existence of c.f.s.]  

Any $\sigma$ extends to a path with initial segments deciding all computable infinitary sentences of complexity at most that of $\varphi$. 

\end{lem} 

\begin{proof}

Let $(\psi_i)_{i\in\omega}$ be a list of the sentences in the forcing language.  Starting with $\sigma_0 = \sigma$, we build a c.f.s.\ $\sigma_0\subseteq\sigma_1\subseteq\ldots$, where $\sigma_{i+1}$ decides $\psi_i$.  We choose $\sigma_i$ at a level of at least $i$, so the sequence determines a path.    
\end{proof}  

\begin{lem} [Truth-and-forcing]  

Let $(\sigma_i)_{i\in\omega}$ be a c.f.s.\ giving the generic path $p$, and let $\mathcal{A}^p$ be the resulting structure.  Then for all $\psi(\bar{x})$ in the forcing language, $\mathcal{A}^p\models\psi(\bar{x})$ iff there is some $i$ such that $\sigma_i\Vdash\psi(\bar{x})$. 

\end{lem} 

\begin{proof}

First, suppose $\psi(\bar{x})$ is finitary quantifier-free.  Suppose the variables $\bar{x}$ are among the first $i$.  If $\mathcal{B}\models\psi(\bar{x})$, then $\sigma_i\Vdash\psi(\bar{x})$.  If $\mathcal{B}\models\neg{\psi(\bar{x})}$, then $\sigma_i\Vdash\neg{\psi(\bar{x})}$, where $\neg{\psi(\bar{x})} = neg(\psi(\bar{x}))$.  No $\sigma_j$ can force $\psi(\bar{x})$.  
Next, let $\psi(\bar{x}) = \bigvee_j(\exists\bar{u}_j)\psi_j(\bar{x},\bar{u}_j)$, where the statement holds for $\psi_i(\bar{x},\bar{z})$.  If $\mathcal{B}\models\psi(\bar{x})$, then for some $i$ and $\bar{z}$, $\mathcal{B}\models\psi_i(\bar{x},\bar{z})$.  Some $\sigma_i$ forces $\psi_j(\bar{x},\bar{z})$, so it forces $\psi(\bar{x})$.  If $\mathcal{B}\models neg(\psi(\bar{x}))$, then for all $j$ and all $\bar{z}$, $\mathcal{B}\models neg(\psi_j(\bar{x},\bar{z}))$.  Then no $\sigma_i$ forces $\psi(\bar{x})$.  Therefore, some $\sigma_i$ forces $neg(\psi(\bar{x}))$ by Lemma \ref{lem:density}.  
This is enough, since we cannot have $\psi(\bar{x})$ and $neg(\bar{x})$ both forced by terms in the c.f.s.  Finally, let $\psi(\bar{x}) = \bigwedge_j(\forall\bar{u}_j)\psi_j(\bar{x},\bar{u}_j)$.  We have already seen that whichever of $\psi(\bar{x})$, $neg(\psi(\bar{x}))$ is true in $\mathcal{B}$ is forced by some $\sigma_i$     
\end{proof} 

\noindent
\textbf{Remark}:  Suppose $(\sigma_i)_{i\in\omega}$ is a chain of forcing conditions that decides all computable $\Sigma_\beta$ sentences for $\beta < \alpha$.  Let $p$ be the resulting path, and let $\mathcal{A}^p$ be the corresponding structure.  Then Truth-and-Forcing holds for computable $\Sigma_\beta$ and computable $\Pi_\beta$ sentences $\psi(\bar{x})$.  

\begin{lem} [Definability of forcing]  

For each sentence $\theta(\bar{x})$ in the forcing language, we can find a formula $force_{\theta(\bar{x})}(y)$ that defines in $T$ the set of $\sigma\in T$ such that $\sigma\Vdash\theta$.  The formula has the same complexity as $\theta(\bar{x})$, except that if $\theta$ is finitary quantifier-free, we may take $force_\theta(y)$ to be either computable $\Sigma_1$ or computable $\Pi_1$.  

\end{lem} 

\begin{proof}

First, suppose $\alpha(\bar{x})$ is finitary quantifier-free.  We have
$\sigma\Vdash\alpha(\bar{x})$ if $\sigma$ satisfies the disjunction of $U_D(y)$ for $D$ such that the variables include those of $\alpha$ and $D$ implies $\alpha$.  This is computable $\Sigma_1$.  We may also say that no extension of $\sigma$ forces $\neg{\alpha}$.  This is computable $\Pi_1$.
Next, consider a computable $\Sigma_\beta$ sentence $\theta(\bar{x}) = \bigvee_i(\exists\bar{u}_i)\psi_i(\bar{x},\bar{u}_i)$.  We have $\sigma\Vdash\theta(\bar{x})$ if the variables of $\sigma$ include $\bar{x}$ and $t(\bar{u}_i)$ for some $i$ and some assignment $t$ taking $\bar{u}_i$ to $X$, and $\sigma\Vdash\psi_i(\bar{x},t(\bar{u}_i))$.  Using the induction hypothesis, we get a computable $\Sigma_\beta$ formula saying $\sigma\Vdash\theta(\bar{x})$.  
Finally, consider a computable $\Pi_\beta$ sentence $\theta(\bar{x}) = \bigwedge_i(\forall\bar{u}_i)\psi_i(\bar{x},\bar{u}_i)$.  We have $\sigma\Vdash\theta(\bar{x})$ if no extension of $\sigma$ forces the $\Sigma_\beta$ sentence $neg(\theta(\bar{x}))$.  This is computable $\Pi_\beta$, again by induction.
\end{proof} 

\begin{lem}

Let $\mathcal{A}$ be an $L$-structure with a computable $\Pi_\alpha$ Scott sentence $\varphi$.  Let $p$ be a generic path through $T_\mathcal{A}$.  Then the structure given by $p$ is isomorphic to $\mathcal{A}$.  

\end{lem}

\begin{proof}  

Take $\sigma\prec p$ deciding $\varphi$.  We can extend $\sigma$ to a c.f.s.\ $(\sigma_i)_{i\in\omega}$ such that the tuple $\bar{a}_i$ represented by $\sigma_i$ includes the first $i$ elements of $\mathcal{A}$.  Let $q$ be the path with $\sigma_i\prec q$ for all $i$.  Then $\mathcal{A}^q\cong\mathcal{A}$, so $\mathcal{A}^q\models\varphi$,  By Truth-and-Forcing, $\sigma\Vdash\varphi$, so $\mathcal{A}^p\models\varphi$.  
\end{proof} 

\begin{defn} [Weak forcing]

For $\sigma\in T$ and $\psi(\bar{x})$ a sentence in our forcing language, $\sigma$ \emph{weakly forces} $\psi(\bar{x})$ if no extension of $\sigma$ forces $neg(\psi)$.

\end{defn} 

\noindent
\textbf{Remark}:  The lemma above shows that if $\varphi$ is a computable $\Pi_\alpha$ Scott sentence for $\mathcal{A}$, then the base node of $T_\mathcal{A}$ weakly forces $\varphi$.  For computable $\Pi_\alpha$ formulas, weak forcing is the same as forcing.     

\bigskip

\begin{lem}

Suppose $\sigma\in T_\mathcal{A}$, where $\sigma$ has variables $\bar{x}$ and $\sigma$ represents the tuple $\bar{a}$.  For any computable infinitary formula $\psi(\bar{x})$ satisfied by $\bar{a}$, $\sigma$ weakly forces $\psi(\bar{x})$.  

\end{lem}
 
\begin{proof}

Suppose $\tau$ is an extension of $\sigma$ that decides $\psi(\bar{x})$.  There is a c.f.s.\ $(\sigma_i)_{i\in\omega}$, extending $\tau$, and such that $\sigma_i$, with variables $\bar{x}_i$, represents a tuple $\bar{a}_i$ including the first $i$ elements of $\mathcal{A}$.  Let $q$ be the generic path with $\sigma_i\prec q$ for all $i$.  Then $\mathcal{A}^q\cong\mathcal{A}$, with an isomorphism $F$ that maps $\bar{x}_i$ to $\bar{a}_i$.  Since $\mathcal{A}\models\psi(\bar{a})$, $\mathcal{A}^q\models\psi(\bar{x})$, so $\tau\Vdash\psi(\bar{x})$.  Then $\sigma$ weakly forces $\psi(\bar{x})$.  
\end{proof}

So far, we have used only the simplest properties of structures in $\mathbb{T}$.  Below, we use forcing to state some further properties.

\bigskip
\noindent
\textbf{$\alpha$-agreement}:  Let $\psi(\bar{x})$ be a computable $\Sigma_\alpha$ formula, where $\bar{x}$ is among the first $n$ variables.  For $\sigma$ at level $n$, if some extension $\tau$ of $\sigma$ forces $\tau\Vdash\psi(\bar{x},\bar{u})$, then $\sigma$ weakly forces $\psi(\bar{x},\bar{u})$.

\bigskip
\noindent
If $\alpha$ is a computable limit ordinal, then we consider just the disjuncts of $\psi(\bar{x})$, each of which is computable $\Sigma_\beta$ for some $\beta < \alpha$.  We see that $T$ satisfies $\alpha$-agreement iff it satisfies $\beta$-agreement for all $\beta < \alpha$. 

\bigskip
\noindent
\textbf{$\alpha$-permutation}:  Let $\sigma$ be a node with variables $\bar{x}$.  Let $\psi(\bar{x},z)$ be a computable $\Sigma_\alpha$ formula, where $z$ is not in $\bar{x}$.  If some $\tau\supseteq\sigma$ forces $\psi(\bar{x},z)$, and $x'$ is the first variable not in $\bar{x}$, then some $\tau'\supseteq\sigma$ forces $\psi(\bar{x},x')$.  (With $\alpha$-agreement, this means that the initial segment of $\tau'$ with variables $\bar{x},x'$ weakly forces $\psi(\bar{x},x')$.) 

\bigskip
\noindent
Again, if $\alpha$ is a computable limit ordinal, then we consider the disjuncts, each of which is computable $\Sigma_\beta$ for some $\beta < \alpha$.  We see that $T$ satisfies $\alpha$-permutation iff it satisfies $\beta$-permutation for all $\beta < \alpha$. 

\bigskip

Below, we give computable infinitary axioms for all of the basic properties.

\begin{lem}\label{lem:property'sAxioms}

For a computable ordinal $\alpha\geq 2$, there is a computable $\Pi_\alpha$-sentence stating the properties of Length, Consistency, Replication, plus $\beta$-agreement, and $\beta$-permutation for $\beta < \alpha$.  

\end{lem}

\begin{proof}

\noindent
\textbf{Length}:  Recall that $L$ is a finite relational language.  For each $n$, there are only finitely many possible formulas $D$ assigning an atomic type to the tuple $\bar{x} = (x_1,\ldots,x_n)$.  

\noindent
(1)  First, taking into account all $n$, we say that each $x$ satisfies $U_D$ for exactly one $D$.  This is computable $\Pi_2$.  To say that $x = \lambda$ (the base node), we write $p(x) = x$.  To say that $x$ is at level $n\geq 1$, 
we say $p^{n+1}(x) = p^n(x) \ \&\ p^n(x) \not= p^{n-1}(x)$.  This is finitary quantifier-free.           

\noindent
(2)  We say that for each $n$, for all $x$ at level $n$, $x$ satisfies $U_D$ for one of the finitely many $D$  assigning an atomic type to $(x_1,\ldots,x_n)$.  This is computable~$\Pi_2$.

\noindent
Altogether, it is computable $\Pi_2$ to say that $T$ satisfies (1) and (2).      

\bigskip
\noindent
\textbf{Consistency}:  Recall that $T$ has the Consistency Property if for $\sigma\subseteq\tau$, the atomic type $D'$ assigned by $\tau$ implies the atomic type $D$ assigned by $\sigma$.  For a given $n$, let $P_n$ be the set of pairs $(D,D')$ such that $D$ is an atomic type in the first $n$ variables, $D'$ is an atomic type in the first $n+1$ variables, and $D'$ logically implies $D$.  We say for all $x,y$ if $y$ is at level $n+1$ and $p(y) = x$, then $\bigvee_{(D,D')\in P_n} (U_Dx\ \&\ U_{D'}y)$.  This is computable $\Pi_2$.    

\bigskip
\noindent
\textbf{Replication}:  Recall that $T$ has the Replication Property if for all $\sigma$ if some successor satisfies $U_D$, then infinitely many do.  This is computable $\Pi_2$.

\bigskip
\noindent
\textbf{$\beta$-agreement}:  Recall that $T$ has the $\beta$-agreement Property if for all $n$, for all computable $\Sigma_\beta$ formulas $\psi(\bar{x})$ in variables among the first $n$, and all $\sigma$ at level $n$, if some $\tau\supseteq\sigma$ has $\tau\Vdash\psi(\bar{x})$, then no $\tau'\supseteq\sigma$ forces $neg(\psi(\bar{x}))$. 
It is computable $\Sigma_\beta$ to say of $\tau$ that it extends $\sigma$ and forces $\psi(\bar{x})$.  It is computable $\Pi_{\beta}$ to say of $\tau\supseteq\sigma$ that it forces $neg(\psi(\bar{x}))$.  To say that $\tau$ does not force $neg(\psi(\bar{x}))$ is computable $\Sigma_\beta$.  We have a computable $\Pi_{\beta+1}$ sentence saying for all $n$, for all $\psi(\bar{x})$ with variables among the first $n$ and $\sigma$ at level $n$, if some extension of $\sigma$ forces $\psi(\bar{x})$, then no extension forces $neg(\psi(\bar{x}))$.    

\bigskip
\noindent
 If $\alpha = \beta+1$, then it is enough to say that $T$ has the $\beta$-agreement property, which is $\Pi_{\beta+1}$.  If $\alpha$ is a computable limit ordinal, then we say that $T$ has the $\beta$-agreement property for all $\beta < \alpha$.  This is computable $\Pi_\alpha$.

\bigskip
\noindent
\textbf{$\beta$-permutation}:  Recall that $T$ has the $\beta$-permutation property if for $\psi$ a computable $\Sigma_\beta$ formula, and $\bar{x}$ consisting of the first $n$ variables, if $\sigma$, at level $n$, has an extension forcing $\psi(\bar{x},z)$, where $z$ is a later variable, then some extension of $\sigma$ forces $\psi(\bar{x},x')$, where  $x'$ is the $(n+1)^{st}$ variable.  For fixed $n$, $\psi$ and $\sigma$ at level $n$, it is computable $\Sigma_\beta$ to say that some extension of $\sigma$ forces $\psi(\bar{x},z)$ for some later variable $z$, and it is computable $\Sigma_\beta$ to say that some immediate extension forces $\psi(\bar{x},x')$.  Then it is computable $\Pi_{\beta+1}$ to say that for all $n$, $\psi$, and $\sigma$ at level $n$, if some extension of $\sigma$ forces $\psi(\bar{x},z)$ for some $z$, then some immediate extension of $\sigma$ forces $\psi(\bar{x},x')$.

\bigskip
\noindent
 If $\alpha = \beta+1$, it is enough to say that $T$ has the $\beta$-permutation property, which is computable $\Pi_{\beta+1}$.  If $\alpha$ is a computable limit ordinal, then we say that $T$ has the $\beta$-permutation property for all $\beta < \alpha$.  This is computable $\Pi_\alpha$.          
\end{proof}

\subsection{Completing the proof of Proposition \ref{prop:PushForwardCompSentences}}

 We suppose that $\varphi$ is a computable $\Pi_\alpha$ Scott sentence for $\mathcal{A}$.  Necessarily, $\alpha\geq 2$.  By Lemma \ref{lem:property'sAxioms}, we have a computable $\Pi_\alpha$ sentence $\theta$ characterizing the labeled trees $T$ that satisfy the basic properties Length, Consistency, Replication, plus $\beta$-agreement and $\beta$-permutation, for all $\beta < \alpha$.  Let $\varphi_*$ be the natural computable $\Pi_\alpha$-sentence saying that the root node weakly forces $\varphi$.  Then $\theta\ \&\ \varphi_*$ is a computable $\Pi_\alpha$ sentence.  The result below says that this is a Scott sentence for $T_\mathcal{A}$.  This is all we need to complete the proof of Proposition \ref{prop:PushForwardCompSentences}.  

\begin{prop}\label{prop:finishing}

Suppose that $\mathcal{A}$ has a computable $\Pi_\alpha$ Scott sentence.  Let $T$ be a labeled tree satisfying $\theta\ \&\ \varphi_*$.  Then $T_\mathcal{A}\cong T$.

\end{prop} 

\begin{proof}

By Theorem \ref{thm:AKM}, the fact that $\mathcal{A}$ has a computable $\Pi_\alpha$ Scott sentence implies that the orbit of each tuple $\bar{a}$ in $\mathcal{A}$ is defined by a computable $\Sigma_\beta$ formula for some $\beta < \alpha$.  For each $\bar{a}$, we choose (non-effectively) such a formula $d_{\bar{a}}(\bar{x})$.  Let $D_{\bar{a}}(\bar{x})$ be the formula giving the atomic type of $\bar{a}$.  To show that $T\cong T_\mathcal{A}$, we define a back-and-forth family.  We use a preliminary notion. 

\begin{defn}  

Let $\sigma$ be a node at level $n$ in $T_\mathcal{A}$, representing a tuple $\bar{a}$ in $\mathcal{A}$.  We say that $\tau$, at level $n$ of $T$, is a \emph{good match for $\sigma$} if $\tau$ 
satisfies $U_{D_{\bar{a}}}$ and weakly forces $d_{\bar{a}}(\bar{x})$. 

\end{defn} 

\begin{defn}

Let $F$ be the set of partial isomorphisms $f$ that take a finite labeled subtree $S$ of $T_\mathcal{A}$ to a finite labeled subtree $S'$ of $T$ such that for each node $\sigma\in S$, $f(\sigma)$, in $S'$ is a good match for $\sigma$.  

\end{defn}

We prove that $F$ has the back-and-forth property.   

\medskip
\noindent
\textbf{$\mathbf{F\not=\emptyset}$}.  The base node $\lambda$ in $T_\mathcal{A}$ represents $\emptyset$ in $\mathcal{A}$.  We may take $d_\emptyset$ to be $\varphi$, the Scott sentence for $\mathcal{A}$.  The base node $\lambda$ in $T$ is the unique good match, as by assumption $T\models\varphi_*$ so $\lambda$ weakly forces $\varphi$.  The function $f$ taking the base node $\lambda$ of $T_\mathcal{A}$ to the base node $\lambda$ of $T$ is in $F$.    

\medskip
\noindent
\textbf{Forth}.  Take $f\in F$, with domain $S$ and range $S'$.  Let $\sigma\in T_\mathcal{A}$ be a terminal node in $S$, with variables $\bar{x}$, and let $\sigma'$ be a successor of $\sigma$, with variables $\bar{x},x'$.  Say that $\sigma$ represents $\bar{a}$, and $\sigma'$ represents $\bar{a},a'$.  Suppose that $f(\sigma) = \tau$, where $\tau$ is a good match for $\sigma$.  We need a successor $\tau'$ of $\tau$ that is a good match for $\sigma'$.  Since $\tau$ weakly forces $\varphi\ \&\ d_{\bar{a}}(\bar{x})$, some extension $\nu$ of $\tau$ forces this.  Any generic path $p$ extending $\nu$ gives a structure $\mathcal{A}^p$ satisfying $\varphi\ \&\ d_{\bar{a}}(\bar{x})$, so $(\mathcal{A}^p,\bar{x})\cong (\mathcal{A},\bar{a})$.  Let $\nu'$ be the extension of $\nu$ including an $y$ element corresponding to $a'\in(\mathcal{A},\bar{a})$ and note that $\nu$ weakly forces $d_{\bar{a},a'}(\bar{x},y)$.  
By the Permutation Property, there is a successor $\tau'$ of $\tau$ such that $\tau'$ weakly forces $d'_{\bar{a,a}}(\bar{x},x')$, where $x'$ is the next variable after~$\bar{x}$.  We extend $f$, taking $\sigma'$ to $\tau'$.     

\medskip
\noindent
\textbf{Back}.  Take $f$, $\sigma$, $\tau$ $\bar{a}$ and $\bar{x}$ be as in the previous paragraph.  Suppose $\tau'$ is a successor of $\tau$.  There is a generic path $p$ extending $\tau'$.  We have an isomorphism $g$ from $\mathcal{A}$ to $\mathcal{A}^p$ that takes $\bar{a}$ to $\bar{x}$.  Say that $x'$ is the next variable after $\bar{x}$.  Then $\tau'$ has variables $\bar{x},x'$.  Let $a'$ be the $g$-pre-image of $x'$.  For the formula $d_{\bar{a},a'}$ that defines the orbit of $\bar{a},a'$ in $\mathcal{A}$, some initial segment of $p$ must force $d_{\bar{a},a'}(\bar{x},x')$, and $\tau'$ must weakly force it.  Let $\sigma'$ be the successor of $\sigma$ that represents $\bar{a},a'$.  Then $\tau'$ is a good match for $\sigma'$.  We extend $f$, taking $\sigma'$ to $\tau'$.      
\end{proof}

\medskip
\noindent
\textbf{Remark}: If $\mathcal{A}$ has a Scott sentence $\varphi$ that is computable $\Sigma_\alpha$, or computable $d$-$\Sigma_\alpha$, the sentence $\theta \& \varphi_*$ is still $\Pi_{\alpha+1}$ because $\theta$ is $\Pi_{\alpha+1}$. This is consistent with the results in \cite{GR}. They show, in the boldface setting, that if $\mathcal{A}$ has Scott complexity  $\Sigma_\alpha$ and $d$-$\Sigma_\alpha$, the $T_\mathcal{A}$ has Scott complexity $\Pi_{\alpha+1}$. This means that the above result is the best possible. For example, consider $\omega^\alpha\cdot 2$ where $\alpha$ is a computable ordinal. This structure has a computable $d$-$\Sigma_{\alpha+1}$ Scott sentence and no $\Pi_{\alpha+1}$ Scott sentence (see \cite{Mbook2} Chapter 2). Because it has no $\Pi_{\alpha+1}$ Scott sentence, $T_{\omega^\alpha\cdot 2}$ has Scott complexity $\Pi_{\alpha+2}$. Therefore, the argument above produces a computable Scott sentence of the ideal complexity, $\Pi_{\alpha+2}$. Note that this precludes the possibility of a simpler set of tree properties $\theta$. If $\theta$ was $d$-$\Sigma_\alpha$, $\theta\&\varphi_*$ would yield a contradictory $d$-$\Sigma_\alpha$ Scott sentence for $T_{\omega^\alpha\cdot 2}$.

\section{Describing $\mathbb{T}^\alpha$}

Recall that for a countable ordinal $\alpha$, $\mathbb{T}^\alpha$ is the class of labeled trees $T$ such that for some $\mathcal{A}\in Mod(L)$ of Scott rank at most $\alpha$, $T\cong T_\mathcal{A}$.  In this section, we show that if $\alpha$ is a computable ordinal, then $\mathbb{T}^\alpha$ is effective $\Pi_{2\alpha+2}$.  This relativizes, and we get the fact that for all countable ordinals $\alpha$, $\mathbb{T}^\alpha$ is $\mathbf{\Pi_{2\alpha+2}}$ in the Borel hierarchy.  We will see that a labeled tree $T$ is in $\mathbb{T}^\alpha$ iff it satisfies certain properties, all of which are given by computable $\Pi_{2\alpha+2}$-sentences.

More specifically, we will show that a labeled tree $T$ is in $\mathbb{T}^\alpha$ iff it satisfies the following properties:  

\begin{enumerate}

\item  levels

\item  consistency

\item  replication

\item  $\alpha$-agreement

\item  $\alpha$-permutation

\item  $\alpha$-support

\end{enumerate}

In the previous section, we stated properties characterizing the labeled $T$ isomorphic to $T_\mathcal{A}$, for $\mathcal{A}$ having a specific computable infinitary Scott sentence.  In that section, we used forcing.  The first three properties---levels, consistency, and replication---were stated in a way that did not involve forcing.  They were given by computable $\Pi_2$-sentences.  The other properties were defined in terms of forcing.  In this section, we shall give new definitions, which do not involve forcing.  We begin with some back-and-forth relations.            

\begin{defn}

Let $T$ be a labeled tree satisfying levels, consistency, and replication.  Suppose $\sigma$ and $\tau$ be nodes, in $T$, and let $\bar{x}$ and $\bar{y}$ be tuples of variables, of the same length, where $\bar{x}$ is among the variables of $\sigma$ and $\bar{y}$ is among the variables of $\tau$.      

\begin{enumerate}

\item  $(\sigma,\bar{x})\leq_0(\tau,\bar{y})$ if the atomic type given to $\bar{x}$ by the label on $\sigma$ matches that given to $\bar{y}$ by the label on $\tau$,

\item  for $\alpha > 0$, $(\sigma,\bar{x})\leq_\alpha(\tau,\bar{y})$ if for all $\beta < \alpha$, all $\tau'$ extending $\tau$, with variables including $\bar{y}$ and $\bar{y}'$, there exist $\sigma'$ extending $\sigma$ and $\bar{x}'$ such that $(\tau',\bar{x},\bar{y})\leq_\beta(\sigma',\bar{x},\bar{x}')$.

\end{enumerate} 

\end{defn} 
  
In the previous section, we described the trees $T$ isomorphic to $T_\mathcal{A}$ for a specific $\mathcal{A}$, assuming that $\mathcal{A}$ has a computable $\Pi_\alpha$ Scott sentence. In this case, the orbit of each tuple is defined by a computable $\Sigma_\beta$ formula for some $\beta < \alpha$.  We stated the properties of $\alpha$-agreement and $\alpha$-permutation in terms of forcing, which depended on the computability of these orbit-defining formulas.  Here we focus on structures of Scott rank at most $\alpha$.  We do not require that the orbits be defined by computable infinitary formulas.  For this reason, we define $\alpha$-agreement and $\alpha$-permutation in a way that does not involve forcing.  We also add a final property, $\alpha$-support.    

\bigskip
\noindent
\textbf{$\alpha$-agreement}:  A labeled tree $T$ satisfies \emph{$\alpha$-agreement} provided that for all nodes $\nu$, $\sigma$, and $\tau$, and all $\beta < \alpha$, if $\nu\prec\tau$ and $\nu\prec\sigma$, where $\nu$ has variables $\bar{x}$ and $\tau$ has variables $\bar{x},\bar{y}$, then there exists $\tau'$, with variables including $\bar{x}$ and some $\bar{z}$ such that $\sigma\prec\tau'$ and $(\tau,\bar{x},\bar{y})\leq_\beta(\tau',\bar{x},\bar{z})$.      

\bigskip
\noindent
\textbf{$\alpha$-permutation}:  A labeled tree $T$ satisfies \emph{$\alpha$-permutation} provided that for all nodes $\sigma$, with variables $\bar{x}$, all $\beta < \alpha$, if for some $\tau$, with variables $\bar{x},\bar{y},z$, we have $\sigma\prec\tau$, then for $x'$ the next variable after $\bar{x}$, there exists $\tau'$, with variables $\bar{x}$, $x'$ and some $\bar{u}$, such that $(\tau,\bar{x},\bar{y},z)\leq_\beta(\tau',\bar{u},x')$.    

\bigskip
\noindent
\textbf{$\alpha$-support}:  A labeled tree $T$ satisfies \emph{$\alpha$-support} provided that for all $\sigma$ with variables $\bar{x}$, there exist $\beta < \alpha$ and $\tau$, with variables $\bar{x}$ and some $\bar{y}$, such that for all $\tau'$, with variables including $\bar{x}'$ of the same length as $\bar{x}$ and $\bar{y}'$ of the same length as $\bar{y}$, $(\tau,\bar{x},\bar{y})\leq_\beta(\tau',\bar{x}',\bar{y}')$ implies $(\sigma,\bar{x})\leq_\alpha(\tau',\bar{x}')$.   

\bigskip

It is easy to see that all $T\in \mathbb{T}^\alpha$ satisfy $\alpha$-agreement, $\alpha$-permutation, and $\alpha$-support.  It is also easy to see the following.  

\begin{lem}
\label{lem20}

For each computable ordinal $\alpha$, there is a computable $\Pi_{2\alpha+2}$ sentence characterizing the labeled trees $T$ that satisfy $\alpha$-agreement and $\alpha$-permutation (without forcing), and $\alpha$-support.   

\end{lem}

We want to prove the following.  

\begin{thm}
\label{thm21}

If $T$ is a labeled tree satisfying length, consistency, replication, the new versions of $\alpha$-agreement and $\alpha$-permutation, and $\alpha$-support, then $T$ is in $\mathbb{T}^\alpha$.   

\end{thm}

Here is an outline of the proof. 

\begin{enumerate}

\item  We say what it means for a path to be ``$\alpha$-generic'' in a way that does not involve forcing.  

\item  We show that $T$ has an $\alpha$-generic path $p$.

\item  We show that the resulting structure $\mathcal{B} = \mathcal{A}^p$ has rank at most $\alpha$.  

\item  Finally, we show that $T\cong T_\mathcal{B}$.   

\end{enumerate} 

Having fixed a countable ordinal $\alpha$, we say that a labeled tree $T$ has \emph{all the properties}, meaning that it has all six properties listed in the statement of Theorem \ref{thm21}.  
\begin{defn}

Let $T$ be a labeled tree satisfying all the properties.  A path $p$ is \emph{$\alpha$-generic} if for all $\beta < \alpha$ and all pairs of nodes $(\nu,\tau)$, where $\nu$ has variables $\bar{x}$, $\tau$ has variables $\bar{x},\bar{y}$, $\nu\prec p$, and $\nu\prec\tau$, there exists $\tau'$, with variables  including $\bar{x}$ and some $\bar{z}$, such that $\nu\prec\tau'\prec p$, and $(\tau,\bar{x},\bar{y})\leq_\beta (\tau',\bar{x},\bar{z})$. 

\end{defn}

\begin{lem} 
\label{lem22}

Let $T$ be a labeled tree satisfying all the properties.  For each $\sigma\in T$, there is an $\alpha$-generic path with $\sigma\prec p$.   

\end{lem}

\begin{proof}

We build $p$ satisfying the following witnessing requirements.

\bigskip
\noindent
\textbf{$W_{(\nu,\tau,\beta)}$}:  If $\nu\prec p$ and $\nu\prec\tau$, where $\nu$ has variables $\bar{x}$ and $\tau$ has variables $\bar{x},\bar{y}$, then there is some $\tau'\prec p$ with variables including $\bar{x}$ and some $\bar{z}$, such that
$(\tau,\bar{x},\bar{y})\leq_\beta(\tau',\bar{x},\bar{z})$.

\bigskip

Fix a list of the requirements.  To build $p$, we form an increasing chain of nodes $(\sigma_s)_{s\in\omega}$.  Let $\sigma_0 = \sigma$.  Given $\sigma_s$, we take the first requirement $W_{(\nu,\tau,\beta)}$ needing attention, where this means that $\nu\prec\sigma_s$, and there does not exist $\tau'\prec\sigma_s$ as above.  The $\alpha$-agreement property gives an extension $\sigma_{s+1}$ of $\sigma_s$ satisfying the requirement.  The resulting path $p$ is $\alpha$-generic.  
\end{proof}

\begin{lem}
\label{lem23}

Let $T$ be a tree satisfying all the properties, and let $P$ be the set of pairs $(\sigma,\bar{x})$ such that $\sigma\in T$ and $\bar{x}$ is among the variables of $\sigma$.  For all $\beta\leq\alpha$, for all $(\sigma,\bar{x})$ and $(\tau,\bar{y})$ in $P$, if $(\sigma,\bar{x})\leq_\beta(\tau,\bar{y})$, then for all $\alpha$-generic paths $p$ extending $\sigma$ and $q$ extending $\tau$, 
$(\mathcal{A}^p,\bar{x})\leq_\beta(\mathcal{A}^q,\bar{y})$.   

\end{lem}

\begin{proof}

We proceed by induction on $\beta$.  For $\beta = 0$, the label on $\sigma$ determines the atomic type of $\bar{x}$ in $\mathcal{A}^p$, and the label on $\tau$ determines the atomic type of $\mathcal{A}^q$, so the statement holds.  
Supposing that the statement holds for $\delta < \beta$, we show that it holds for $\beta$.  Take $(\sigma,\bar{x})$ and $(\tau,\bar{y})$ in $P$, where $(\sigma,\bar{x})\leq_\beta(\tau,\bar{y})$.  Let $p$ and $q$ be $\alpha$-generic paths, where $\sigma\prec p$ and $\tau\prec q$.  We must show that $(\mathcal{A}^p,\bar{x})\leq_\beta(\mathcal{A}^q,\bar{y})$.  
Take $\delta < \beta$ and $\bar{v}$, and let $\tau'\prec q$, with variables including $\bar{y},\bar{v}$.  We need $\sigma'\prec p$ with variables including $\bar{x}$ and some $\bar{u}$, such that 
$(\tau',\bar{y},\bar{v})\leq_\delta (\sigma',\bar{x},\bar{u})$. By the definition of the back-and-forth relations, there exists $\sigma*$, with variables including 
$\bar{x}$ and some $\bar{u}*$, such that 
$\sigma\prec\sigma*$ and
$(\tau,\bar{y},\bar{v})\leq_\delta(\sigma*,\bar{x},\bar{u}*)$.  By $\alpha$-genericity, there exists $\sigma'\prec p$, with variables including $\bar{x}$ and some $\bar{u}$, such that $(\sigma*,\bar{x},\bar{u}*)\leq_\delta(\sigma',\bar{x},\bar{u}')$.  Then
$(\tau',\bar{y},\bar{v})\leq_\delta(\sigma',\bar{x},\bar{u}')$.  By the Induction Hypothesis, $(\mathcal{A}^q,\bar{y},\bar{v})\leq_\delta(\mathcal{A}^p,\bar{x},\bar{u}')$.  Then
$(\mathcal{A}^p,\bar{x})\leq_\beta(\mathcal{A}^q,\bar{y})$, as required.          
\end{proof}  

\begin{lem}
\label{lem24}

Suppose $T$ satisfies all the properties, and let $p$ be an $\alpha$-generic path.  Then the resulting structure $\mathcal{A}^p$ has Scott rank at most $\alpha$. 

\end{lem}

\begin{proof}  

Recall that $\mathcal{A}^p$ has universe equal to $X$.  By Lemma \ref{addedcondition}, to show that $\mathcal{A}^p$ has Scott rank at most $\alpha$, it is enough to show that for each $\bar{x}$ in $X$, there exist $\bar{u}$ and $\beta < \alpha$ such that for all $\bar{x}'$ of the same length as $\bar{x}$ and $\bar{u}'$ of the same length as $\bar{u}$, 
$(\mathcal{A}^p,\bar{x},\bar{u})\leq_\beta(\mathcal{A}^p,\bar{x}',\bar{u}')$ implies 
$(\mathcal{A}^p,\bar{x})\leq_\alpha(\mathcal{A}^p,\bar{x}')$.  
Take $\sigma\prec p$ with variables including $\bar{x}$.  By $\alpha$-support, there exist $\beta < \alpha$, and $\tau$, with variables including $\bar{x}$ and some $\bar{u}$, such that 
for all $\tau'$ with variables including $\bar{x}'$ the same length as $\bar{x}$ and $\bar{u}'$ the same length as $\bar{u}$, $(\tau,\bar{x},\bar{u})\leq_\beta(\tau',\bar{x}',\bar{u}')$ implies $(\tau,\bar{x})\leq_\alpha(\tau',\bar{x}')$.  By $\alpha$-agreement, there exists $\tau*$, with variables $\bar{x}$ and some $\bar{u}^*$ such that $\nu\prec\tau*\prec p$ and 
$(\tau,\bar{x},\bar{u})\leq_\beta(\tau*,\bar{x},\bar{u}^*)$.  Then $(\tau*,\bar{x},\bar{u}^*)\leq_\beta(\tau',\bar{x}',\bar{u}')$ implies $(\tau*,\bar{x})\leq_\alpha(\tau',\bar{x}')$.  By Lemma \ref{lem23}, $(\mathcal{A}^p,\bar{x})\leq_\alpha(\mathcal{A}^p,\bar{x}')$.  
\end{proof}  

\begin{lem}
\label{lem:alphaToAlphaPlus1}

Suppose $T$ satisfies all the properties, and let $(\sigma,\bar{x})\leq_{\alpha}(\tau,\bar{y})$,  Then $(\sigma,\bar{x})\leq_{\alpha+1}(\tau,\bar{y})$.

\end{lem}    

\begin{proof}
    Assume that $(\sigma,\bar{x})\leq_{\alpha}(\tau,\bar{y})$.  
    Take an extension $(\tau',\bar{y}')$ of $(\tau,\bar{y})$.
    We aim to show that $(\sigma,\bar{x})$ has an extension $(\sigma',\bar{x}')$ such that $(\sigma',\bar{x}')\geq_{\alpha}(\tau',\bar{y}')$.
    By the definition of $\alpha$-support, there exist an extension $(\tau'',\bar{y}',\bar{u})$ of $(\tau,\bar{y}')$ and $\beta<\alpha$ such that for all $(\nu,\bar{z},\bar{w})$,  $(\tau'',\bar{y}',\bar{u})\leq_\beta(\nu,\bar{z},\bar{w})$ implies $(\tau',\bar{y}')\geq_\alpha(\nu,\bar{z})$.
    Since $(\sigma,\bar{x})\leq_\alpha(\tau,\bar{y})$, $(\sigma,\bar{x})$ has an extension $(\sigma'',\bar{x}',\bar{u}')$ with $(\sigma'',\bar{x}',\bar{u}')\geq_\beta(\tau'',\bar{y}',\bar{u})$.
    By $\alpha$-support, $(\sigma',\bar{x}')\geq_{\alpha}(\tau',\bar{y}')$.
    Thus, $(\sigma,\bar{x})\leq_{\alpha+1}(\tau,\bar{y})$ as desired.
\end{proof}

\begin{lem}
\label{lem25}

If $T$ satisfies all the properties, and $p$ and $q$ are $\alpha$-generic paths through $T$.  Then $\mathcal{A}^p\cong\mathcal{A}^q$.     

\end{lem}

\begin{proof}

Let $\mathcal{F}$ be the set of finite partial isomorphisms that take $\bar{x}\in p$ to $\bar{y}\in q$, where for some $\sigma\prec p$, with variables $\bar{x}$ and $\tau\prec q$, with variables $\bar{y}$ we have $(\sigma,\bar{x})\leq_\alpha(\tau,\bar{y})$.  We show that $\mathcal{F}$ has the back-and-forth property.

\bigskip

Clearly, $\mathcal{F}$ is non-empty.  

\bigskip
\noindent
\textbf{Forth}:  Suppose $\bar{x}\rightarrow\bar{y}$ is in $\mathcal{F}$.  Say $\sigma\prec p$ have variables $\bar{x}$ and $\tau\prec q$ have variables $\bar{y}$ with $(\sigma,\bar{x})\leq_\alpha(\tau,\bar{y})$. 
By Lemma \ref{lem:alphaToAlphaPlus1}, $(\sigma,\bar{x})\leq_{\alpha+1}(\tau,\bar{y})$.
Another application of Lemma \ref{lem:alphaToAlphaPlus1} gives that $(\sigma,\bar{x})\geq_{\alpha+1}(\tau,\bar{y})$.
 Let $z$ be a variable not in $\bar{x}$ and $\sigma'$ be an extension of $\sigma$ with $\sigma'\prec p$ including $z$.
 By the definition of the back-and-forth relations, there is an extension $(\tau',\bar{y},w)$ of $(\tau,\bar{y})$,(not necessarily along $q$) such that  $(\tau',\bar{y},w)\geq_\alpha(\sigma',\bar{x},z)$.
 In the definition of $\alpha$-support, there is an extension $(\tau'',\bar{y},w,\bar{u})$ and ordinal $\beta<\alpha$ associated to $(\tau',\bar{y},w)$ such that for all $(\nu,\bar{s},t,\bar{r})$, $(\tau'',\bar{y},w,\bar{u})\leq_\beta (\nu,\bar{s},t,\bar{r})$ implies that $(\tau',\bar{y},w)\geq_\alpha (\nu,\bar{s},t)$.
 By $\alpha$-genericity of $q$, there is an extension $(\mu,\bar{y},w^*,\bar{u}*)$ of $(\tau,\bar{y})$ along $q$ such that $(\mu,\bar{y},w^*,\bar{u}*)\geq_\beta(\tau'',\bar{y},w,\bar{u})$.
 This means that $(\mu,\bar{y},w^*)\leq_\alpha (\tau',\bar{y},w)$.
By Lemma \ref{lem:alphaToAlphaPlus1}, $(\mu,\bar{y},w^*)\geq_\alpha (\tau',\bar{y},w)$.
This yields that $(\mu,\bar{y},w^*)\geq_\alpha (\sigma',\bar{x},z)$, and so $\bar{x},z\rightarrow\bar{y},w^*$ is in $\mathcal{F}$.

\bigskip
\noindent
\textbf{Back}:  Suppose $\bar{x}\rightarrow\bar{a}$ is in $\mathcal{F}$.
Say $\sigma\prec p$ have variables $\bar{x}$ and $\tau\prec q$ have variables $\bar{y}$ with $(\sigma,\bar{x})\leq_\alpha(\tau,\bar{y})$.
By Lemma \ref{lem:alphaToAlphaPlus1}, $(\sigma,\bar{x})\geq_{\alpha}(\tau,\bar{y})$.
The argument in the Forth section yields for any $(\tau',\bar{y},z)$ extending $(\tau,\bar{y})$ along $q$, there is a $(\mu,\bar{x},w^*)$ extending $(\sigma,\bar{x})$ along $p$.
So, $\bar{x},w^*\rightarrow\bar{y},z$ as desired.
\end{proof}

To complete the proof of Theorem \ref{thm21}, it is enough to show the following.

\begin{lem}
\label{lem26}

Let $T$ satisfy all of the properties, and let $p$ be an $\alpha$-generic path.  If $\mathcal{B} = \mathcal{A}^p$, then $T\cong T_\mathcal{B}$.  
\end{lem}

\begin{proof}   

For an $\alpha$-generic path $q$ through $T$, we write $\mathcal{A}^q$ for the resulting structure.  For an $\alpha$-generic path $r$ through $T_\mathcal{B}$, we write $\mathcal{B}^r$ for the resulting structure.  By Lemma \ref{lem25}, these structures are all isomorphic.  Moreover, in either tree, the structures obtained from $\alpha$-generic paths that extend a specific node assign the same $\Pi_\alpha$-type to the variables of that node.  With this in mind, we define an isomorphism $F$ from $T$ onto $T_\mathcal{B}$, preserving the tree structure, and with the special feature that for $F(\sigma) = \tau$, where both nodes are at level $n$ and for the tuple $\bar{x}$ consisting of the first $n$ variables, the $\Pi_\alpha$-type of $\bar{x}$ in structures $\mathcal{A}^q$ for $\alpha$-generic paths $q$ through $T$ that extend 
$\sigma$, matches the $\Pi_\alpha$-type of $\bar{x}$ in $\mathcal{B}^r$ for $\alpha$-generic paths $r$ through $T_\mathcal{B}$ that extend $\tau$.  Let $\mathcal{F}$ be the set of partial $1-1$ functions $f$ mapping a finite subtree $S$ of $T$ onto a finite subtree $S'$ of $T_\mathcal{B}$, such that $f$ preserves the tree structure and has the further special feature.  We show that $\mathcal{F}$ has the back-and-forth property.  

\bigskip
\noindent     
\textbf{$F\not=\emptyset$}:  The function taking the root node of $T$ to the root node of $T_\mathcal{B}$ has the special feature, since for $\alpha$-generic paths $q$ through $T$ and $r$ through $T_\mathcal{B}$, $\mathcal{A}^q\cong\mathcal{B}^r$.  

\bigskip
\noindent
\textbf{Forth}:  Suppose $f$ is in $\mathcal{F}$, mapping $S$ onto $S'$.  Let $\sigma'$ be a further node in $T$.  We may suppose that $\sigma'$ is a successor to some node $\sigma$ of $S$.  Say that $F(\sigma) = \tau$.  Let $\bar{x}$ be the variables of $\sigma$.  For all $\alpha$-generic paths $q$ extending $\sigma$ and all $\alpha$-generic paths $r$ extending $\tau$, $(\mathcal{A}^q,\bar{x})\cong(\mathcal{B}^r,\bar{x})$ under an isomorphism $F_{q,r}$.  Let $x'$ be the next variable after $\bar{x}$. Let $r$ be an $\alpha$-generic path extending $\tau$.  Then $F_{q,r}$ takes $\bar{x}$ to $\bar{x}$ and takes $x'$ to some $y$.  Let $\tau'$ be the successor of $\tau$ with $\tau'\prec q$.  By $\alpha$-permutation, $\tau$ has another successor $\tau*$ such that $(\tau',\bar{x},y)\leq_\alpha(\tau^*,\bar{x},x')$.  If $r^*$ is an $\alpha$-generic extension of $\tau^*$, then 
$(\mathcal{A}^q,\bar{x},x')\cong(\mathcal{B}^r,\bar{x},x')$.  We let $f'$ be the extension of $f$ that takes $\sigma'$ to $\tau*$.  This is in $\mathcal{F}$. 

\bigskip
\noindent
\textbf{Back}:  Suppose $f$ is in $\mathcal{F}$, mapping the subtree $S$ of $T$ onto the subtree $S'$ of $T_\mathcal{B}$, as above, and let $\tau'$ be a further node of $T_\mathcal{B}$, where $\tau'$ is a successor of some $\tau$ in $S'$. Say that $f(\sigma) = \tau$.  Proceeding just as above, we find a successor $\sigma^*$ of $\sigma$ such that the extension of $f$ taking $\sigma^*$ to $\tau'$ is in $\mathcal{F}$.  
\end{proof}  


\section{Sharpness of the Bound}

In the previous section, we showed that if $\alpha\geq 1$ is an $X$-computable ordinal, then $\mathbb{T}^\alpha$ is $X$-effective $\Pi_{2\alpha+2}$.  In this section, we show that the bound is sharp.  

\begin{thm}\label{thm:sharpness}

For an $X$-computable ordinal $\alpha$, the class $\mathbb{T}^\alpha$ is complete $X$-effective $\Pi_{2\alpha+2}$.

\end{thm}

What drives the complexity of $\mathbb{T}^\alpha$ is the Scott rank.  Let $L$ be the language of graphs, and let $Mod(L)^\alpha$ consist of the structures in $Mod(L)$ of Scott rank at most $\alpha$.  It is well-known (see \cite{Mbook2} Lemma II.67) that for $\alpha\geq 1$ an $X$-computable ordinal, $Mod(L)^\alpha$ is $X$-effective $\Pi_{2\alpha+2}$.  To prove Theorem \ref{thm:sharpness}, we prove the following.

\begin{thm}
\label{Mod(L)}

For all $X$, and all $X$-computable ordinals $\alpha\geq 1$, $Mod(L)^\alpha$ is $\Pi_{2\alpha+2}$-hard, within $Mod(L)$.  

\end{thm}

Before proving Theorem \ref{Mod(L)}, we show how it yields Theorem \ref{thm:sharpness}.

\begin{proof} [Proof of Theorem \ref{thm:sharpness}, assuming Theorem \ref{Mod(L)}]

Recall that the Friedman-Stanley embedding is Turing computable and preserves Scott rank by \cite{GR}.  Take a class $D\subseteq 2^\omega$ that is $X$-effective $\Pi_{2\alpha+2}$.  By Theorem \ref{Mod(L)}, there is an $X$-computable operator $\Phi:2^\omega\rightarrow Mod(L)$ such that  for all $f\in 2^\omega$, $f\in D$ iff $\Phi(f)\in Mod(L)^\alpha$.  Composing $\Phi$ with the Friedman-Stanley embedding, we get an $X$-computable operator $\Psi:2^\omega\rightarrow\mathbb{T}$ such that $f\in D$ iff $\Psi(f)\in \mathbb{T}^\alpha$.  
\end{proof}    


The remainder of this section is devoted to the proof of Theorem \ref{Mod(L)}.  We use ideas from \cite{HKL}.  

\begin{defn}

Let $\gamma$ be an $X$-computable ordinal.  

\begin{enumerate}

\item  $h^X_{\Sigma_\gamma}(f)$ is the set of indices for $X$-effective $\Sigma_\gamma$ sets that contain $f$.  

\item  $h^X_{\Pi_\gamma}$ is the set of indices for $X$-effective $\Pi_\gamma$ sets that contain $f$.

\end{enumerate}
\end{defn}

The following observation is stated in \cite{HKL}.  It is analogous to the result of Ash saying that for $\alpha$ an $X$-computable ordinal, the set of (indices for) $X$-computable $\Sigma_\gamma$ ($\Pi_\gamma$) sentences true in a structure $\mathcal{A}$ is $\Sigma^0_\gamma$ ($\Pi^0_\gamma$) relative to $X$ and $\mathcal{A}$, with all imaginable uniformity.  The proof, like that of Ash, is a straightforward induction on $\gamma$.  

\begin{prop}
\label{Ash} 

For $X\subseteq\omega$, $\gamma\geq 1$ an $X$-computable ordinal, and $f\in 2^\omega$,  $h^X_{\Sigma_\gamma}(f)$ is $\Sigma^0_\gamma$ relative to $(X,f)$, and $h^X_{\Pi_\gamma}(f)$ is $\Pi^0_\gamma$ relative to $(X,f)$, with all imaginable uniformity.  

\end{prop}

Below, for each $X$, we fix a path through $\mathcal{O}^X$ and identify each $X$-computable ordinal $\alpha$ with its unique notation on this path.  

\bigskip
\noindent
\textbf{Theorem 4.2}.  Given a set $X$ and an $X$-computable ordinal $\alpha$, we can compute a function $\Phi$ that takes takes each pair $(f,b)$, where $f\in 2^\omega$ and $b$ is an index for an $X$-effective $\Pi_{2\alpha+2}$ set $B$, to a graph $\mathcal{A}^{(X,f)}_d$ such that $SR(\mathcal{A}^{(X,f)}_b)\leq\alpha$ iff $f\in B$. 

\bigskip 

Given an index $b$ for an $X$-effective $\Pi_{2\alpha+2}$ set $D$, we get, effectively, a sequence of indices $(d_n)_{n\in\omega}$ for $X$-effective $\Sigma_{2\alpha+1}$ sets $D_n$ such that $B = \cap_n 
D_n$.  We will prove the following.  

\bigskip
\noindent
\textbf{Technical Theorem}.  Given a set $X$ and an $X$-computable ordinal $\alpha$, we have an $X$-computable operator $\Psi$ that takes each pair $(f,d)$, where $f\in 2^\omega$ and $d$ is an index for an $X$-effective $\Sigma_{2\alpha+1}$ set $D$, to a graph $\mathcal{B}^{(X,f)}_d$ such that $SR(\mathcal{B}^{(X,f)}_d)\leq\alpha$ iff $f\in D$.   

\bigskip

Before proving the Technical Theorem, let us see how it gives Theorem 4.2.

\begin{proof} [Proof of Theorem 4.2 from Technical Theorem]

Given $X$, $f$, and $(d_n)_{n\in\omega}$, we obtain the structure $\mathcal{A}$ from a sequence of structures $\mathcal{B}_n$, computable uniformly in $X$, $f$, and $n$, such that $\mathcal{B}_n$ has rank at most $\alpha$ iff $f\in D_n$.  We pass effectively from the sequence $(\mathcal{B}_n)_{n\in\omega}$ to a single structure $\mathcal{B}$ for the language obtained by adding to $L$ unary relation symbols $U_n$, for $n\in\omega$.  In $\mathcal{B}$, the relations $U_n$ partition the universe of $\mathcal{B}$ into infinite sets, and the binary relation $E$ puts a copy of $\mathcal{B}_n$ on the set $U_n$.  The rank of $\mathcal{B}$ is at most $\alpha$ iff all $\mathcal{B}_n$ have rank at most~$\alpha$.  
By Proposition 1.10, we can pass effectively from $\mathcal{B}$ to a graph $\mathcal{A}$, preserving isomorphism and Scott rank.  
\end{proof}  

We split the proof of the Technical Theorem into four subsections: Subsection 4.1 deals with 
$\alpha = 1$, Subsection 4.2 deals with $\alpha$ finite and greater than $1$, Subsection 4.3 deals with limit $\alpha$, and Subsection 4.4 deals with $\alpha = \lambda + n$, for $\lambda$ a limit ordinal and $n > 0$ finite.   

\subsection{The case $\alpha=1$}

In this sub-section, we prove the following.        

\bigskip
\noindent
\textbf{Technical Theorem for $\alpha = 1$}.  Given a set $X$, we have an $X$-computable operator $\Psi$ that takes each pair $(f,d)$, where $f\in 2^\omega$ and $d$ is an index for an $X$-effective $\Sigma_3$ set $D$, to a graph $\mathcal{B}^{(X,f)}_d$ such that $SR(\mathcal{B}^{(X,f)}_d) \leq 1$ iff  $f\in D$ ($e\in h^X_{\Sigma_3}(f)$).        
 
\begin{defn}

For $X\subseteq\omega$ and $f\in 2^\omega$, $Cof^{(X,f)}$ is the set of $n$ such that $W_n^{(X,f)}$ is co-finite.  

\end{defn}

The standard proof that $Cof$ is $\Sigma^0_3$-hard has a great deal of uniformity.  Given an index $d$ for a $\Sigma^0_3$ set $S_d\subseteq\omega$ and $k\in\omega$, we can effectively find a number $n$ such that $W_n$ is co-finite iff $k\in S_d$.  This is also true when we relativize.  The indices $d$ for sets $S$ that are $\Sigma^0_3$ relative to $(X,f)$ form a computable set $I_{\Sigma_3}$, which is the same for all $(X,f)$.  Given $d\in I$ and $k\in\omega$, we can compute $n$ such that for all $X\subseteq\omega$ and all $f\in 2^\omega$, $W_n^{(X,f)}$ is co-finite iff $k\in S_d$, where this is the set with index $d$ as a set $\Sigma^0_3$ relative to $(X,f)$.  Thus, to prove the Technical Theorem for $\alpha = 1$, it is enough to prove the following.

\begin{lem}
\label{Cof}  

There is a Turing operator taking each triple $(X,f,n)$, where $X\subseteq\omega$, $f\in 2^\omega$, and $n\in\omega$, to a graph $\mathcal{A}^{(X,f)}_n$, a bunch of daisies, such that $n\in Cof^{(X,f)}$ iff $\mathcal{A}^{(X,f)}_n$ has rank at most $1$.    

\end{lem}

Recall that a \emph{daisy} consists of a \emph{center} that is the single common point of cycles of various sizes.  The cycles are called \emph{petals}.  A \emph{bunch} of daisies is a graph whose connected components are daisies.  

\begin{proof}  

We build $\mathcal{A}^{(X,f)}_n$, following a uniform procedure based on $W^{(X,f)}_n$.  We enumerate the centers of the daisies as $c_i$ for $i\in\omega$.  If $i\in W_n^{(X,f)}$, then $c_i$ will have one petal of length $k+3$ for each $k\in W_n^{(X,f)}$.  If $i\notin W_n^{(X,f)}$, then $c_i$ will have one petal of length $i+3$, and one of length $k+3$ for each $k < i$ such that $k\in W_n^{(X,f)}$.  This is all---the daisy with center $c_i$ has no further petals, and each graph element belongs to one of these daisies.  At stage $s$, we have built a finite graph, with centers $c_i$ for $i < s$ and petals based on elements enumerated into $W_n^{(X,f)}$ by stage $s$.  In the limit, we get $\mathcal{A}^{(X,f)}$ with a completed daisy around $c_i$, just as described.  

For simplicity, we suppose that $0\in W_n^{(X,f)}$.  This means that for $i\notin W_n^{(X,f)}$, $c_i$ has at least two petals, one of length $3$ and one of length $i+3$.  Then we can define the set of centers by a simple existential formula saying that $x$ is connected by the edge relation to at least $3$ points.  Without the simplifying assumption, we may have a center $c_i$ with just one petal, all elements of which are automorphic, a case we wish to avoid.
We must show that $n\in Cof^{(X,f)}$ iff $\mathcal{A}^{(X,f)}$ has rank at most $1$, where this means that the orbits of all tuples are defined by existential formulas.       

First, suppose $n\in Cof^{(X,f)}$.  There are finitely many $i$ such that $i\notin W_n^{(X,f)}$.  For each such $i$, $c_i$ is the only center with a petal of size $i+3$.  We have an existential definition of $c_i$.  All $c_j$ for $j\in W_n^{(X,f)}$ are automorphic.  We want an existential definition for the common orbit of these $c_j$.  Take $b$ greater than all $j\notin W_n^{(X,f)}$.  Then to define the common orbit of $c_i$ for $i\in W_n^{(X,f)}$, we say that it is a center with a petal of size $b+3$.  For a tuple $\bar{a}$ on the daisy with center $c_i$, we have an existential formula $\psi(\bar{x},c_i)$, describing a finite set of petals that contains $\bar{a}$ and saying how the elements $a_i$ sit on these petals, in relation to $c_i$.  The orbit of an arbitrary tuple $\bar{a}_1,\ldots,\bar{a}_m$, with $\bar{a}_j$ in the daisy with center $c_i$ is defined by an existential formula $\varphi(\bar{x}_f,\ldots,\bar{x}_ m)$ saying that there exist $y_j$ in the orbit of $c_j$ such that $\bigwedge_{1\leq j\leq m}\psi(\bar{x}_j,y_j)$.    

Now, suppose $n\notin Cof^{(X,f)}$.  Consider any existential formula $\varphi(x)$ true of $c_i$ for $i\in W_n^{(X,f)}$.  Fix one such $c_i$.  The truth of $\varphi(x)$ is witnessed by a finite subgraph $G$ consisting of $c_i$, with finitely many of its petals, and $G'$ consisting of finite parts of the daisies around finitely many other centers $c_j$.  Take $b\notin W_n^{(X,f)}$ such that $b+3$ is greater than the size of any petal in $G$.  Let $G^*$ be the modification of $G$ consisting of $G'$ and the daisy with center $c_b$ and all of its petals.  Since $i < b$, all petals around $c_i$ in $G$ are matched by petals around $c_b$.  There is an embedding of $G$ in $\mathcal{A}^{(X,f)}_n$ that takes $c_i$ to $c_b$, fixing $G'$ pointwise.  Thus, in $\mathcal{A}^{(X,f)}_n$, $c_b$ satisfies $\varphi(x)$.  Since $b\notin W_n^{(X,f)}$, $c_b$ has a petal of size $b+3$, while $c_i$ does not, so $\varphi(x)$ does not define the orbit.  
\end{proof}


\subsection{The case $n > 1$ finite}

Our goal in this subsection is to prove the following.

\bigskip
\noindent
\textbf{Technical Theorem for finite $n > 1$}.  
Given a set $X$, we have an $X$-computable operator $\Psi$ that takes each pair $(f,d)$, where $f\in 2^\omega$ and $d$ is an index for an $X$-effective $\Sigma_{2n+1}$ set $D$, to a graph $\mathcal{B}^{(X,f)}_d$ such that $SR(\mathcal{B}^{(X,f)}_d)\leq n$ iff $f\in D$.   

\bigskip
\noindent
\textbf{Outline of proof}:  The structure $\mathcal{B}^{(X,f)}_d$ will be last in a sequence  
$\mathcal{B}_1,\mathcal{B}_2,\ldots,\mathcal{B}_n$ such that

\begin{itemize}

\item [(a)]  $\mathcal{B}_1$ is computable in $X^{2(n-1)}$ and $f$, with rank $1$ if $f\in D$ and $2$ otherwise,  

\item [(b)]  for $1\leq i < n$, $\mathcal{B}_{i+1}$ is computable in $(X,f)^{2(n-i)}$, with rank equal to $1+SR(\mathcal{B}_i)$.

\end{itemize}

Given $X$, $f$, and the index $d$ for an $X$-effective $\Sigma_{2n+1}$ set $D$, we compute an index $d_1$ for $\mathcal{B}_1$ that is $\Delta^0_{2(n-1)+1}$ relative to $(X,f)$.
 For $1\leq i < n$, from $d_i$ an index for $\mathcal{B}_i$ that is $\Delta^0_{2(n-i)+1}$ relative to $(X,f)$, we compute an index $d_{i+1}$ for $\mathcal{B}_{i+1}$ that is $\Delta^0_{2(n-(i+1))+1}$ relative to $X$ and $f$.  
 We arrive at an index for $\mathcal{B}^{(X,f)}_d = \mathcal{B}_n$ computable in $(X,f)$.  The rank of $\mathcal{B}_n$ is $n$ if $f\in D$ and $n+1$ otherwise.  To complete the proof of the Technical Theorem for finite $n > 1$, we flesh out items (a), computing the index for $\mathcal{B}_1$,  and (b), passing from the index for $\mathcal{B}_i$ to the index for $\mathcal{B}_{i+1}$.  
Note that being $\Delta^0_{2(n-i)+1}$ relative to $(X,f)$ is the same as being computable in $(X,f)^{2(n-i)}$.   

\bigskip
\noindent
For (a), we recall Section 4.1, where we described a uniform effective procedure that takes each triple $(X,f,d)$, where $X\subseteq\omega$, $d$ is an index for an $X$-effective $\Delta_3$ set $D$, and $f\in 2^\omega$, to a graph $\mathcal{A}^{(X,f)}_d$ that has Scott rank $1$ if $f\in D$ and $2$ otherwise.  Recall that $h^X_{\Sigma_3}(f)$ is the set of indices for $X$-effective $\Sigma_3$ sets that contain $f$.  This set is $\Sigma^0_3$ relative to $(X,f)$, with a known index $b$.  In Section~4.1, we construct $\mathcal{A}^{(X,f)}_d$ so that it has rank $1$ if $d$ is in the set with index $b$ as a set $\Sigma^0_3$ relative to $(X,f)$.  The construction shows the following.

\begin{prop}[Technical Theorem for $\alpha = 1$, re-stated]  

Given $(Y,b,n)$, where $b$ is an index for a set $S$ that is $\Sigma^0_3$ relative to $Y$, we can effectively find an index for a graph $\mathcal{A}$ computable relative to $Y$, such that 
\[SR(\mathcal{A}) = \left\{\begin{array}{cc}
1 & \mbox{if $n\in S$}\\
2 & \mbox{otherwise}
\end{array}\right. .\] 

\end{prop}         

Let $d$ be an index for an $X$-effective $\Sigma_{2n+1}$ set $D$, and let $f\in 2^\omega$.  The set $h_{\Sigma_{2n+1}}(f)$ is $\Sigma^0_{2n+1}$ relative to $(X,f)$, with known index, and it is $\Sigma^0_3$ relative to $(X,f)^{2(n-1)}$, also with known index.  Apply the proposition above, letting $Y = (X,f)^{(2(n-1))}$, letting $S = h_{\Sigma^0_{2n+1}}(f)$, letting $b$ be the known index for $S$ as a set $\Sigma^0_3$ relative to $Y$, and letting $n$ be $d$.  By the proposition, we can find an index $b_1$ for a graph $\mathcal{B}_1$, computable relative to $(X,f)^{2(n-1)}$, such that $\mathcal{B}_1$ has rank $1$ if $d\in h^X_{\Sigma_{2n+1}}(f)$ (which means $f\in D$), and $2$ otherwise.     

\bigskip
\noindent
(b)  For $1\leq i < n$, to pass from an index $b_i$ for $\mathcal{B}_i$ computable in $(X,f)^{2(n-i)}$ to an index $b_{i+1}$ for $\mathcal{B}_{i+1}$ computable in $(X,f)^{2(n-(i+1))}$, we use Harrison-Trainor's ``unfriendly jump inversion,'' adding some further information.  In \cite{HTarith}, Harrison-Trainor proved the following. 

\begin{thm}\label{thm:MHT}
    Let $S\subseteq\omega$ be a $\Sigma_3^0$ set. There is a uniformly computable sequence of graphs $\mathcal{G},\mathcal{H},\mathcal{C}_i$ such that
    \begin{enumerate}
        \item $i\in S \iff\mathcal{C}_i\cong\mathcal{G}$
        \item $i\not\in S \iff\mathcal{C}_i\cong\mathcal{H}$
        \item $\mathcal{H}\not\leq_2\mathcal{G}$
        \item $SR(\mathcal{G}) = 1$, $SR(\mathcal{H}) = 2$.
        \item $\mathcal{G}\equiv_1\mathcal{H}$
    \end{enumerate}
\end{thm}

\noindent
\textbf{Remark}:  Item (5) above is not explicitly stated in \cite{HTarith}, but it is immediate from the construction given there.

\bigskip

Using Theorem \ref{thm:MHT}, relativized to $Y$, Harrison-Trainor showed that given an index $a$ for a graph $\mathcal{A}$ computable in $Y''$ and of finite Scott rank, we can find an index $b$ for a structure $\mathcal{B}$ computable in $Y$ such that $SR(\mathcal{B}) \leq 1+SR(\mathcal{A})$.  There is a copy of $\mathcal{A}$ defined in $\mathcal{B}$---the formula defining the universe is quantifier-free, and there are $\Sigma_2$ formulas defining the edge relation and its negation.  However, looking more closely at the construction, we will see that $SR(\mathcal{B}) = 1+SR(\mathcal{A})$.    

\begin{thm}

Given $Y$ and an index $a$ for a graph $\mathcal{A}$ computable in $Y''$, and of finite rank, we can find an index $b$ for a structure $\mathcal{B}$ computable in $Y$ such that $SR(\mathcal{B}) = 1+SR(\mathcal{A})$.  

\end{thm}  

By Proposition 1.10, we can effectively replace the original $\mathcal{B}$ by a graph of the same rank.  The structure $\mathcal{B}$ (in its initial version) has unary predicates $U,V$ that partition the universe into disjoint infinite sets.  The set $U$ will correspond to the universe of a copy of $\mathcal{A}$, but the edge relation $E$ is coded.  For simplicity, we identify $U$ with the universe of $\mathcal{A}$.  The structure $\mathcal{B}$ has a unary function $c$ from $V$ from $V$ to the set of triples $(x,y,i)$, where $x,y$ are distinct elements of $U$ and $i\in 2$.\footnote{To be precise, we may think of relations $c_0(v,x,y)$, which holds if $c(v) = (x,y,0)$, and $c_1(v,x,y)$, which holds if $c(v) = (x,y,1)$.}  For each such triple, $(x,y,i)$, the $c$-inverse image is an infinite set $V_{(x,y,i)}$.  Finally, the  structure $\mathcal{B}$ has a binary relation $R$ that, on $V_{(x,y,i)}$ puts a copy of 
$\mathcal{G}$ if $xEy$ and $i = 1$ or if $\neg{xEy}$ and $i = 0$, and $\mathcal{H}$ otherwise.  The relation $R$ is the union of its restrictions to the sets $V_{(x,y,i)}$.        

\bigskip

The fact that from the index for $\mathcal{A}$ computable in $Y''$, we can compute the index for $\mathcal{B}$ computable in $Y$, is clear.  The relations $xEy$ and $\neg{xEy}$ are both $\Sigma^0_3$ in $Y$, with known indices, so we can put copies of $\mathcal{G}$, $\mathcal{H}$ on the sets $V_{(x,y,i)}$ as required.  Since $\mathcal{G}$ has Scott rank $1$, it has a $\Pi_2$ Scott sentence $\varphi$.  This sentence is false in $\mathcal{H}$.  For the copy of $\mathcal{A}$ defined in $\mathcal{B}$, the universe is defined by the quantifier-free formula $U(x)$.  The relation $\mathcal{A}\models xEy$ has a $\Pi_2$ definition saying that $\varphi$ holds in $V_{(x,y,0)}$, and a $\Sigma_2$ definition saying that $\neg{\varphi}$ holds in $V_{(x,y,1)}$.     

We know that if $\mathcal{A}$ has rank $n$, $\mathcal{B}$ has rank at most $1+n$.  We must show that the rank of $\mathcal{B}$ is not less than $1+n$.  It is enough to show that if $\mathcal{A}$ is an infinite graph of rank $n$, not less, then the resulting $\mathcal{B}$ has rank $1+n$, not less.  We start with the case $SR(\mathcal{A}) = 1$.

\begin{lem}

Let $\mathcal{A}$ be a graph of rank $1$, with pairs of distinct elements $(x,y),(x',y')$ such that $xEy$ and $\neg{x'Ey'}$.  Then the structure $\mathcal{B}$ obtained from $\mathcal{A}$ must have rank greater than $1$.  

\end{lem}

\begin{proof}  

Since $\mathcal{G}$ and $\mathcal{H}$ satisfy the same existential sentences, the pairs $(x,y)$ and $(x',y')$ satisfy the same existential formulas. By assumption, these pairs are not automorphic. Therefore, $\mathcal{B}$ must have rank greater than $1$.
\end{proof}
  
For $n\geq 2$, to show that for $\mathcal{A}$ of rank at least $n$, the resulting $\mathcal{B}$ has rank at least $1+n$, we must understand how the back-and-forth relations on $\mathcal{B}$ are related to those on $\mathcal{A}$.  

\begin{lem}
\label{back&forth}

Let $\bar{b} = \bar{a},\bar{v}$, and $\bar{b}' = \bar{a}',\bar{v}'$ be tuples in $\mathcal{B}$, where $\bar{a},\bar{a}'$ are in $\mathcal{A}$, $\bar{v},\bar{v}'$ are in $V$.  Let $\bar{v}_{(x,y,i)}$ be the part of $\bar{v}$ in $V_{(x,y,i)}$, and let $\bar{v}'_{(x',y',i)}$ be the corresponding part of $\bar{v}'$.  We suppose that $x,y\in\bar{a}$ and $x',y'\in\bar{a}'$.
\begin{enumerate}

\item [(a)]  $(\mathcal{B},\bar{a},\bar{v})\leq_1 (\mathcal{B},\bar{a}',\bar{v}')$ iff the equality relation on $\bar{a}$ matches that on $\bar{a}'$, and for all $(x,y,i)$ such that $\bar{v}_{(x,y,i)}$ is non-empty, 
$(V_{(x,y,i)},\bar{v}_{x,y,i})\leq_1(V_{(x',y',i)},\bar{v}_{(x,y,i)})$.  

\item [(b)]  For $m \geq 2$, $(\mathcal{B},\bar{a},\bar{v})\leq_m(\mathcal{B},\bar{a},\bar{v})$ iff $(\mathcal{A},\bar{a})\leq_{m-1}(\mathcal{A},\bar{a}')$ and for each $(x,y,i)$ such that $\bar{v}_{(x,y,i)}$ is non-empty, $(V_{(x,y,i)},\bar{v}_{(x,y,i)})\cong(V_{(x',y',i)},\bar{v}_{(x,y,i)})$.

\end{enumerate}

\end{lem}

\begin{proof}

We proceed by induction on $m$, starting with the base case $m = 1$.  In $\mathcal{B}$, existential formulas say nothing about a tuple $\bar{a}$ in $\mathcal{A}$ beyond the equality relation, and existential statements about a tuple $\bar{v}$ in $V$ splits into separate statements about the tuples $\bar{v}_{(x,y,i)}$ in $V_{(x,y,i)}$.  Thus, $(\mathcal{B},\bar{a},\bar{v})\leq_1(\mathcal{B},\bar{a}',\bar{v}')$ iff the equalities on $\bar{a}$ and $\bar{a}'$ match and for each $(x,y,i)$ such that 
$\bar{v}_{(x,y,i)}$ is non-empty, $(V_{(x,y,i)},\bar{v}_{(x,y,i)})\leq_1(V_{(x',y',i)},\bar{v}_{(x',y',i)})$. 

Assuming that the statement holds for $m$, where $m\geq 1$, we prove it for $m+1$.  Suppose $(\mathcal{B},\bar{a},\bar{v})\leq_{m+1}(\mathcal{B},\bar{a}',\bar{v}')$.  For all $\bar{c}'$ in $\mathcal{A}$, there exists $\bar{c}$ such that 
$(\mathcal{B},\bar{a}',\bar{v}',\bar{c}')\leq_m(\mathcal{B},\bar{a},\bar{v},\bar{c})$.
By inductive hypothesis, $(\mathcal{A},\bar{a}',\bar{c}')\leq_{m-1}(\mathcal{A},\bar{a},\bar{c})$.  Thus, $(\mathcal{A},\bar{a})\leq_m(\mathcal{A},\bar{a}')$.  
If $m=1$, by induction hypothesis, we similarly have that $(V_{(x,y,i)},\bar{v}_{(x,y,i)})\leq_2(V_{(x',y',i)},\bar{v}_{(x',y',i)})$ for each $x,y$.
For one of the $i\in2$,  $V_{(x,y,i)}\cong\mathcal{G}$, which means $(V_{(x,y,i)},\bar{v}_{(x,y,i)})\cong(V_{(x',y',i)},\bar{v}_{(x',y',i)})$ as $\mathcal{G}$ is Scott rank 1.
By construction, this also means that $(V_{(x,y,1-i)},\bar{v}_{(x,y,1-i)})\cong(V_{(x',y',1-i)},\bar{v}_{(x',y',1-i)})$.
If $m>1$, then $(V_{(x,y,i)},\bar{v}_{(x,y,i)})\cong(V_{(x',y',i)},\bar{v}_{(x',y',i)})$ follows immediately from the induction hypothesis.

Now, suppose $(\mathcal{A},\bar{a})\leq_m(\mathcal{A},\bar{a}')$ and for the relevant triples $(x,y,i)$  and corresponding $(x',y',i)$, $(V_{(x,y,i)},\bar{v}_{(x,y,i)})\cong (V_{(x',y',i)},\bar{v}_{(x',y',i)})$.  We show that $(\mathcal{B},\bar{a},\bar{v})\leq_{m+1}(\mathcal{B},\bar{a}',\bar{v}')$.  For $\bar{c}'$ in $\mathcal{A}$ and $\bar{w}'$ in $V$, we need $\bar{c}$ and $\bar{w}$ such that 
$(\mathcal{B},\bar{a}',\bar{c}',\bar{v}',\bar{w}')\leq_m(\mathcal{B},\bar{a},\bar{c},\bar{v},\bar{w})$.  Extending, if necessary, we suppose that each element of $\bar{w}'$ is in $V_{(x',y',i)}$ for some $x',y'\in \bar{a}',\bar{c}'$.  
Since $(\mathcal{A},\bar{a})\leq_m(\mathcal{A},\bar{a}')$, we have $\bar{c}$ such that $(\mathcal{A},\bar{a}',\bar{c}')\leq_{m-1}(\mathcal{A},\bar{a},\bar{c})$.  There is no problem choosing $\bar{w}$; we select $\bar{w}$ according to the isomorphism between $(V_{(x,y,i)},\bar{v}_{(x,y,i)})\cong (V_{(x',y',i)},\bar{v}_{(x',y',i)})$ in each of the relevant $V_{(x,y,i)}$. By inductive hypothesis, this guarantees that $(\mathcal{B},\bar{a}',\bar{c}',\bar{v}',\bar{w}')\leq_m(\mathcal{B},\bar{a},\bar{c},\bar{v},\bar{w})$, as desired.
\end{proof}

The next result will complete the proof of Theorem 4.7.  

\begin{lem}

Let $n\geq 2$.  If $\mathcal{A}$ is a graph of rank at least $n$, then $\mathcal{B}$ has rank at least $n+1$.  

\end{lem}

\begin{proof}

The fact that $\mathcal{A}$ does not have rank at most $n-1$ means that there is some $\bar{a}$ such that for all $\bar{b}$, there exist $\bar{a}',\bar{b}'$ such that 

\begin{enumerate}

\item  $(\mathcal{A},\bar{a},\bar{b})\leq_{n-2}(\mathcal{A},\bar{a}',\bar{b}')$,

\item  $(\mathcal{A},\bar{a})\not\leq_{n-1}(\mathcal{A},\bar{b})$.  

\end{enumerate}
We claim that this same $\bar{a}$ witnesses that $\mathcal{B}$ does not have rank at most $n$.  We must show that for each $\bar{b}$ in $\mathcal{B}$, there exist $\bar{a}',\bar{b}'$ such that 
\begin{enumerate}

\item  $(\mathcal{B},\bar{a},\bar{b})\leq_{n-1}(\mathcal{B},\bar{a}',\bar{b}')$,

\item $(\mathcal{B},\bar{a})\not\leq_n(\mathcal{B},\bar{a}')$.  

\end{enumerate}

Take $\bar{b}$ in $\mathcal{B}$.  Let $\bar{c}$ be the part of $\bar{b}$ in $\mathcal{A}$, and let $\bar{v}$ be the part in $V$.  Extending, if necessary, we suppose that $\bar{v}$ breaks up into tuples $\bar{v}_{(x,y,i)}$ for $x,y\in \bar{a},\bar{c}$.  Since $\bar{a}$ witnesses that $\mathcal{A}$ does not have Scott rank at most $n-1$, we have $\bar{a}',\bar{c}'$ such that 
$(\mathcal{A},\bar{a},\bar{c})\leq_{n-1}(\mathcal{A},\bar{a}',\bar{c}')$ and $(\mathcal{A},\bar{a})\not\leq_n(\mathcal{A},\bar{a}')$.  For each $(x,y,i)$ such that $v_{(x,y,i)}$ is a non-empty tuple in $\bar{v}$, choose $\bar{v}'_{(x',y',i)}$ a corresponding tuple such that $(V_{(x,y,i)},\bar{v}_{(x,y,i)})\cong(V_{(x',y',i)},\bar{v}'_{(x',y',i)})$.  We can do this, since $xEy$ iff $x'Ey'$ by $(\mathcal{A},\bar{a},\bar{b})\leq_{n-2}(\mathcal{A},\bar{a}',\bar{b}')$.  By Lemma 4.9, $(\mathcal{B},\bar{a},\bar{c},\bar{v})\leq_n(\mathcal{B},\bar{a}',\bar{c}',\bar{v}')$ and 
$(\mathcal{B},\bar{a})\not\leq_{n+1}(\mathcal{B},\bar{a}')$.  
\end{proof}  

\subsection{The case $\alpha$ a limit ordinal}

In this subsection, we prove the following.  

\begin{thm}
\label{4.3}

For $X\subset\omega$ and $\alpha$ an $X$-computable limit ordinal, there is an $X$-computable operator that takes pairs $(f,d)$, where $f\in 2^\omega$ and $d$ is an index for an $X$-effective $\Sigma_{\alpha+1}$ set $D\subseteq 2^\omega$, to a graph $\mathcal{A}^{(X,f)}_d$ that has Scott rank at most $\alpha$ iff $f\in D$. 

\end{thm}

\begin{proof}

Given $d$, we can find a sequence of indices $(d_n)$ for $X$-effective $\Pi_{\alpha}$ sets $D_n$ with union $D$.  Let $S = h^X_{\Pi_\alpha}(f)$.  For $Y = (X,f)$, $S$ is $\Pi_\alpha$ relative to $Y$.  From $\alpha$, we compute an increasing sequence of ordinals $(\alpha_n)_{n\in\omega}$ with limit $\alpha$.  By a result from \cite{AK}, for each $n$, we can effectively associate with the index $d_n$ an ordering $\mathcal{L}_n$, computable in $Y$, such that $\mathcal{L}_n$ has type $\omega^\alpha$ if $d_n\in S$ and some $\beta_n < \omega^\alpha$, otherwise.  

The initial version of $\mathcal{A}^{(X,f)}_d$ is a structure $\mathcal{B}$ with an equivalence relation $\sim$ that partitions the universe into disjoint sets $V$ and $U_n$ (not named).  There is a relation that puts an ordering of type $\omega^\alpha$ on $V$ and an ordering of type $\omega^{\alpha_n}\cdot L_n$ on $U_n$. 
We must show that $SR(\mathcal{A}^{(X,f)}_d)\leq\alpha$ iff $f\in D$. 
Note that if $d_n\notin S$, then $U_n$ has type less than $\omega^\alpha$.  
On the other hand, if $d_n\in S$, then $U_n$ has type equal to $\omega^\alpha$.  

Consider the case that $f\in D$.
In particular, for some $j$, for all $i\geq j$ we have that $d_i\in S$.
 This means that the structure $\mathcal{A}^{(X,f)}_d$ will contain infinitely many copies of linear orderings of the form $\omega^{\alpha}$ and finitely many equivalence classes with ordinals strictly less than $\omega^\alpha$.
 Let these finitely many ordinals be denoted $A_0\leq\cdots\leq A_N$.
    Note that, by the ordinal analysis in \cite{Ash86}, there is a sentence $\varphi_N$ in $\Pi_{<\alpha}$ that states that a given ordinal is longer than $A_N$.
    Furthermore, each $A_i$ has $SR(A_i)<\alpha$, so it has a $\Pi_{<\alpha}$ Scott sentence $\psi_i$ and $\omega^\alpha$ has a $\Pi_{\alpha+1}$ Scott sentence $\chi$.
We obtain a $\Pi_{\alpha+1}$ $\mathcal{A}^{(X,f)}_d$ by expressing that for each $i$, there is one equivalence class satisfying $\psi_i$, the rest of the classes satisfy $\varphi_N$ and every equivalence class has $\varphi_N\to \chi$.
This gives, $SR(\mathcal{A}^{(X,f)}_d)\leq\alpha$.

    Consider the case $x\not\in T$.
    This means that for all $j$, $d_i\not\in S$.
    The structure $\mathcal{A}^{(X,f)}_d$ will contain one copy of $\omega^\alpha$ and the rest of the equivalence classes will contain a linear ordering strictly less than $\omega^\alpha$.
    Furthermore, by construction, for each $i$, among these smaller linear orderings, there will be an ordinal $A_i$ that is left divisible by $\omega^{\alpha_i}$.
    Let $b$ be the first element of the $\omega^\alpha$ and $a_i$ be the first element of $A_i$.
    It follows from the ordinal analysis in \cite{Ash86} that $(\omega^\alpha,b)\equiv_{\alpha_i} (A_i,a_i)$.
    In particular, for each $i$, there is no $\Sigma_{\alpha_i}$ description of the automorphism orbit of $b$.
    Taken together, we conclude that there is no $\Sigma_{\alpha}$ description of the automorphism orbit of $b$.
    Therefore, $SR(\mathcal{A}^{(X,f)}_d)>\alpha$.

By Proposition 1.10, we can pass effectively from $\mathcal{B}$ to a graph $\mathcal{A}^{(X,f)}_d$, preserving rank.  
\end{proof}  

\subsection{The case $\alpha = \lambda + n$, for limit $\lambda$ and finite $n\geq 1$}

In this subsection, we prove the following.

\begin{thm}

For $X\subseteq\omega$ and $\alpha = \lambda + n$, where $\lambda$ is an $X$-computable limit ordinal and $n\geq 1$ is finite, there is an $X$-computable operator that takes each pair $(f,d)$, where $f\in 2^\omega$ and $d$ is an index for an $X$-effective $\Sigma_{2\alpha+1}=\Sigma_{\lambda + 2n+1}$ set $D$ to a graph $\mathcal{A}^{(X,f)}_d$ such that $\mathcal{A}^{(X,f)}_d$ has rank at most $\lambda + n$ iff $f\in D$.   

\end{thm}

Here is an outline of the proof.  Recall that $f\in D$ iff $d\in h^X_{\Sigma_{\lambda+2n+1}}(f)$, where $h^X_{\Sigma_{\lambda+2n+1}}(f)$ is a subset of $\omega$ that is $\Sigma^0_{\lambda+2n+1}$ relative to $(X,f)$.  Let $Y = (X,f)$.    

\begin{enumerate}

\item  By the results from Sections 4.1 and 4.2, we can pass effectively from the index $d$ to an index $a$ for a graph $\mathcal{A}$, computable in $Y^{\lambda+1}$, such that $SR(\mathcal{A})\leq n$ iff $f\in D$.  

\item  We consider a pair of $X$-computable structures $\mathcal{G},\mathcal{H}\in Mod(L)$, from \cite{CGHT}, \cite{GHKMMS}, with some special properties different from those in Section 4.2.
Using the pair $\mathcal{G},\mathcal{H}$, we pass effectively from the index $a$ for $\mathcal{A}$, as a structure computable in $Y^{\lambda+1}$, to an index $b$ for a structure $\mathcal{B}$ computable in $Y$.  The construction is quite similar to that in Section 4.2.    

\item  Using results from \cite{CGHT}, with some of the ideas going back to \cite{Maher}, we show that this $\mathcal{B}$ has the desired Scott rank; that is, $SR(\mathcal{B})\leq \lambda+n$ iff ${SR(\mathcal{A})\leq n}$.  

\item  Finally, we apply Proposition 1.10 to pass from $\mathcal{B}$ to the desired graph 
$\mathcal{A}^{(X,f)}_d$.       

\end{enumerate}
Below, we say more about Steps (2) and (3).

\subsubsection{Properties of $\mathcal{G},\mathcal{H}$, construction of $\mathcal{B}$}

Our goal here is to say exactly what properties are needed for $\mathcal{G}$ and $\mathcal{H}$, and then to describe precisely the construction that takes $\mathcal{A}$ to $\mathcal{B}$.  We fix $X$ and an $X$-computable limit ordinal $\lambda$.  In \cite{GHKMMS}, it is shown that there are $\mathcal{G},\mathcal{H}$ in $Mod(L)$ (linear orderings) with the following properties:

\begin{enumerate}

\item  $\mathcal{G},\mathcal{H}$ are both rigid,   

\item  the pair $\mathcal{G},\mathcal{H}$ is ``$\lambda+1$-friendly;'' that is, they are $X$-computable and the relations relations $\leq_\beta$ for $1\leq\beta < \lambda+1$ are uniformly $X$-c.e.,

\item  $\mathcal{G}\equiv_\lambda\mathcal{H}$, 

\item  $\mathcal{G}\not\leq_{\lambda+1}\mathcal{H}$ and $\mathcal{H}\not\leq_{\lambda+1}\mathcal{G}$, 

\item  both $\mathcal{G},\mathcal{H}$ have Scott rank at most $\lambda$.

\end{enumerate}

Given $\mathcal{G}$ and $\mathcal{H}$, we define an $X$-effective operator $\Phi_\lambda$ that takes each graph $\mathcal{A}$ to a structure $\mathcal{B}$, with unary relations $U,V$, a ternary relation $C$, and a binary relation $R$.  As in Section 4.2, $U,V$ partition the universe into infinite sets, and we identify $U$ with $\mathcal{A}$.  The relation $C$ serves to partition $V$ into infinite sets $V_{(x,y)}$, where $(x,y)$ is a pair of distinct elements of $U$. $C(v,x,y)$ means that $v\in V_{(x,y)}$.  The relation $R$ puts on each $V_{(x,y)}$ a copy of $\mathcal{G}$ if $xEy$ and $\mathcal{H}$ otherwise.  Because of item (4), we have $\Sigma_{\lambda+1}$ formulas saying of a pair $(x,y)$ in $U$ that $V_{(x,y)}$ is isomorphic to $\mathcal{G}$, or to $\mathcal{H}$.   

\bigskip

The next lemma follows from the discussion in Chapter X.3 of \cite{Mbook2} regarding the operator $\Phi_\lambda$.  However, we can give a simple proof.  

\begin{lem}

Let $\mathcal{G}$ and $\mathcal{H}$ be $L$-structures, with $\Pi_{\lambda+2}$ Scott sentences $\varphi_G$, $\varphi_H$, and let $\Phi_\lambda$ be the operator described above, taking each 
$\mathcal{A}\in Mod(L)$ to a structure $\mathcal{B}$ in which $\mathcal{G}$ and $\mathcal{H}$ are used to code the binary relation.  The class of structures $\mathcal{B}$ isomorphic to $\Phi_\lambda(\mathcal{A})$, for $\mathcal{A}\in Mod(L)$, is $\Pi_{\lambda+2}$.  

\end{lem}

\begin{proof}

We axiomatize the class by a $\Pi_{\lambda+2}$-sentence $\rho$ saying the following:
\begin{enumerate}

\item  $U$ is infinite,

\item  for all $x\not= y$ in $U$, the corresponding $V_{(x,y)}$, with the restriction of the relation $R$, satisfies one of $\varphi_G$, $\varphi_H$.

\end{enumerate}
It is computable $\Pi_2$ to say that $U$ is infinite.  We have a $\Pi_{\lambda+2}$ formula saying of $x,y$, that if both are in $U$ and $x\not= y$, then $V_{(x,y)}$ satisfies one of the sentences $\varphi_G$ or $\varphi_H$.
\end{proof}

We also describe how to pass from a $Y^{\lambda+1}$-computable code for $\mathcal{A}$ to a $Y$-computable code for $\mathcal{B}$.
This resembles Lemma 5.3 in \cite{CGHT}

\begin{lem}
Given $Y$ and an index $a$ for a graph $\mathcal{A}$ computable in $Y^{\lambda+1}$ we can find an index $b$ for a graph $\mathcal{B}\cong\Phi_\lambda(\mathcal{A})$ computable in $Y$.
\end{lem}

\begin{proof}

Note that $\mathcal{G}$ and $\mathcal{H}$ are $\lambda+1$-friendly and have $\mathcal{G}\equiv_\lambda\mathcal{H}$.
By the pair of structures theorem \cite{AK}, this means that distinguishing $\mathcal{G}$ and $\mathcal{H}$ is $\Delta_{\lambda+1}(Y)$ hard.
In other terms, given an effective $\Delta_{\lambda+1}$ set $S$ there is a $Y$ effective method of constructing $\mathcal{C}_x$ with $\mathcal{C}_x\cong\mathcal{G}$ if $x\in S$ and $\mathcal{C}_x\cong\mathcal{H}$ if $x\not\in S$.

Using $a$, uniformly given $x,y\in\mathcal{A}$, we can find whether $xEy$ in a $Y^{\lambda+1}$-computable manner.
This means that there is a $\Delta_{\lambda+1}(Y)$ set $S$ of pairs $x,y$ with $xEy$.
We find the index $b$ for $\mathcal{B}$ by describing a $Y$ computable method of building $\mathcal{B}$.
For each element in $\mathcal{A}$, we put a corresponding element with predicate $U$ in $\mathcal{B}$.
We ensure that each $x,y\in\mathcal{A}$ has $V_{x,y}\cong \mathcal{C}_{\langle x,y \rangle}$.
The constructed $\mathcal{B}$ is isomorphic to $\Phi_\lambda(\mathcal{A})$, as it places a copy of $\mathcal{G}$ on $V_{x,y}$ when $xEy$ and a copy of $\mathcal{H}$ on $V_{x,y}$ otherwise.
\end{proof}

\subsubsection{Scott rank of $\mathcal{B}$}

Our goal here is to show that $\mathcal{A}$ has Scott rank at most $n$ iff $\mathcal{B}$ has Scott rank at most $\lambda+n$.  We use some known results, given in \cite[Section 5]{CGHT}.   

\begin{lem}

For each $\Pi_k$-sentence $\varphi$, there is a $\Pi_{\lambda+k}$ sentence $\varphi_*$ such that $\mathcal{A}\models\varphi$ iff $\mathcal{B}\models\varphi_*$.

\end{lem}
  
This is clear from the fact that we have $\Sigma_{\lambda+1}$ formulas that define a copy of $\mathcal{A}$ in $\mathcal{B}$. 

\begin{lem}

For each $\Pi_{\lambda+k}$ sentence $\psi$, there is a $\Pi_k$-sentence $\psi^*$ such that 
$\mathcal{B}\models\psi$ iff $\mathcal{A}\models\psi^*$.  

\end{lem}

The argument given in \cite[Section 5]{CGHT} works by noting that, in some sense, $\Phi_\lambda$ is better than $X$-computable.
Without loss of generality, $\Phi_\lambda$ may produce the $\lambda$ jump of $\mathcal{B}$ (defined precisely in \cite{Mbook2} Chapter X).
Any $\Pi_{\lambda+k}$ sentence $\psi$ in $\mathcal{B}$ corresponds to a $\Pi_k$ sentence $\hat\psi$ in $\mathcal{B}^{(\lambda)}$.
From there, the Pullback Theorem \cite{KMV} gives that $\hat\psi$ corresponds to a $\Pi_k$ sentence $\psi^*$ in $\mathcal{A}$.

\bigskip

The result below now follows from a known pattern of reasoning \cite[Section 6.2]{GHT}.  We can give a simple proof.

\begin{lem}

For $\mathcal{G}$, $\mathcal{H}$, and $\Phi$ as described, and $\Phi(\mathcal{A}) = \mathcal{B}$, we have \linebreak $SR(\mathcal{A}) \leq n$ iff $SR(\mathcal{B}) \leq \lambda+n$. 

\end{lem}

\begin{proof}

First, suppose $SR(\mathcal{A})\leq n$.  Then $\mathcal{A}$ has a $\Pi_{n+1}$ Scott sentence $\varphi$.  By Lemma 4.14, there is a $\Pi_{\lambda+n+1}$ sentence $\varphi_*$ such that $\mathcal{A}\models\varphi$ iff $\mathcal{B}\models\varphi_*$.  By Lemma 4.13, there is a $\Pi_{\lambda+2}$ sentence $\rho$ characterizing the class of structures isomorphic to those in the range of the $X$-computable operator $\Phi$ defined on $Mod(L)$.  The sentence $\varphi_*\ \&\ \rho$ is $\Pi_{\lambda+n+1}$, and it is a Scott sentence for $\mathcal{B}$.  Therefore $SR(\mathcal{B})\leq \lambda+n$.  
Now, suppose $SR(\mathcal{B})\leq\lambda+n$.  Then $\mathcal{B}$ has a $\Pi_{\lambda+n+1}$ Scott sentence $\psi$.  By Lemma 4.15, there is a $\Pi_{n+1}$ sentence $\psi^*$ such that $\mathcal{A}\models\psi^*$ iff $\mathcal{B}\models\psi$.  Then $\psi^*$ is a Scott sentence for $\mathcal{A}$.  Therefore, $SR(\mathcal{A})\leq n$.  
\end{proof}

This was the last piece we needed for the proof of Theorem 4.12, and Theorem 4.12 was all we needed to complete the proof of Theorem 4.1.

\end{document}